\documentclass[12pt,reqno]{amsart}
\usepackage{amssymb,amsmath,amsthm,fullpage,enumerate}
\usepackage{xcolor}
\usepackage{biblatex} 
\usepackage{verbatim}
\usepackage{cases}

\usepackage[foot]{amsaddr}

\usepackage[OT2,OT1]{fontenc}

\DeclareFontFamily{OT2}{cmr}{\hyphenchar\font45 }
\DeclareFontShape{OT2}{cmr}{m}{n}{<->wncyr10}{}
\DeclareFontShape{OT2}{cmr}{m}{it}{<->wncyi10}{}
\DeclareFontShape{OT2}{cmr}{m}{sc}{<->wncysc10}{}
\DeclareFontShape{OT2}{cmr}{b}{n}{<->wncyb10}{}
\DeclareFontShape{OT2}{cmr}{bx}{n}{<->ssub*wncyr/b/n}{}

\DeclareFontFamily{OT2}{cmss}{\hyphenchar\font45 }
\DeclareFontShape{OT2}{cmss}{m}{n}{<->wncyss10}{}

\DeclareRobustCommand\cyr{\fontencoding{OT2}\selectfont}
\DeclareTextFontCommand{\textcyr}{\cyr}

\DeclareUnicodeCharacter{0306}{}

\newtheorem{theorem}{Theorem}[section]
\newtheorem{prop}[theorem]{Proposition}
\newtheorem{lem}[theorem]{Lemma}
\newtheorem{corr}[theorem]{Corollary}

\theoremstyle{definition}

\newtheorem{deff}[theorem]{Definition}

\newtheorem{examp}[theorem]{Example}

\theoremstyle{remark}
\newtheorem{comm}[theorem]{Remark}

\newcommand{\R}{\mathbb R}

\newcommand{\C}{\mathbb C}

\newcommand{\p}{\partial}

\newcommand{\z}{\bar z}

\newcommand{\dbar}{\bar\partial}

\newcommand{\bc}{\mathbb{B}}

\newcommand{\sca}{\text{Sc}\,}
\newcommand{\vect}{\text{Vec}\,}

\newcommand{\disk}{\mathbb{D}}

\title{Generalizations of the Dirichlet problem for bianalytic functions}

\author{William L. Blair}
\address{Department of Mathematics\\
  The University of Texas at Tyler\\
  Tyler, TX, 75799\\
  USA}
\email{wblair@uttyler.edu}



\begin{document}

\begin{abstract}
    We prove the Dirichlet problem for second-order iterated Vekua equations, a natural generalization of the Bitsadze equation, is well-posed when the boundary condition is defined as a product of an exponential function and a polynomial on a non-degenerate conic that is not a circumference. Also, we extend this result, as well as other related results for the Dirichlet problem for polyanalytic and generalized analytic functions from the literature, to their analogues for bicomplex differential equations.  
\end{abstract}

\maketitle

\section{Introduction}\label{section: introduction}

    The study of boundary value problems is an active area of mathematical research motivated by the success in using these problems to model phenomena in the physical sciences. In this article, we study the Dirichlet boundary value problem for higher-order iterated Vekua equations. 

    The classical Dirichlet problem seeks a harmonic function, i.e., a solution $w$ to the Laplace equation $\Delta w = 0$, in a domain that agrees with a prescribed function on the boundary. This problem is well-studied and known to be well-posed on planar domains. In \cite{Bitsadze1948}, Bitsadze shows the Dirichlet problem 
    \begin{equation}\label{eq: dbvpintro}
        \begin{cases}
            \frac{\p^2 w}{\p\z^2} = f, & \text{ in } D,\\
            w = g, & \text{ on } \p D,
        \end{cases}
    \end{equation}
    where $D$ is the complex unit disk, $f\equiv 0$, and $g\equiv 0$, has infinitely many solutions. Since this is, in a sense, the simplest case of the problem for the differential equation $\frac{\p^2 w}{\p\z^2} = f$, now called the Bitsadze equation, it follows that one need not look far to find other differential equations where the associated Dirichlet boundary value problem is not as well-behaved. 

    In \cite{Beg04}, Begehr shows that, for $D$ equal to the complex unit disk, \eqref{eq: dbvpintro} does have a unique solution when $f$ and $g$ satisfy certain integral conditions. So, there is evidence that well-posedness of \eqref{eq: dbvpintro} on the complex unit disk is dependent on the choice of nonhomogeneous part and boundary function. Surprisingly, Moreno Garc\'ia, et al., in \cite{Dirichletbianalytic}, show that when the boundary condition is a polynomial on a non-degenerate conic, that is not a circumference, then \eqref{eq: dbvpintro} is well-posed. That circumferences are excluded agrees with Bitsadze's result on the unit circle discussed above.

    The recent discovery in \cite{Dirichletbianalytic} of boundary sets that guarantee well-posedness of \eqref{eq: dbvpintro} for complex second-order partial differential equations that are not the Laplace equation motivates the work of this article. Specifically, we seek another differential equation that may take the place of the Bitsadze equation, other than the Laplace equation, such that the associated Dirichlet problem is well-posed with the same boundary set as the problem for the Bitsadze equation from \cite{Dirichletbianalytic}. Since it is known that bianalytic functions, i.e., solutions of the Bitsadze equation, and solutions to the second-order iterated Vekua equations studied in \cite{itvek, itvekbvp, metahardy, WB3, WBD, WBD2, BCHoiv} (see \eqref{eq: hoivintro} below) are similar in structure, it follows that these boundary sets are a good candidate for well-posedness for the Dirichlet problem for solutions to second-order Vekua equations. The Dirichlet problem for higher-order iterated Vekua equations was first considered in \cite{WBD2}, and in that work, the author and B. B. Delgado show that the Dirichlet problem for higher-order Vekua equations is not well-posed on the disk, just as the problem for the Bitsadze equation is not, and show that solutions do exist on the disk when certain integral conditions are satisfied, as is the case for the Bitsadze equation in \cite{Beg04}. In Section \ref{sec: 2.2.2} of this paper, we confirm that non-degenerate conics which are not a circumference do provide boundary sets where the Dirichlet problem for a second-order Vekua equation is well-defined, so long as the boundary function is the product of a polynomial and a certain exponential function that is determined by the specific equation.

    The bicomplex numbers are a four real dimensional (or two complex dimensional) extension of the complex numbers first considered by Segre \cite{Segre}. Recently, the bicomplex numbers have found application in analysis, especially in the context of Vekua equations as they provide a direction for studying the complex stationary Schr\"odinger equation. For several examples where bicomplex numbers are used in analysis, see \cite{ComplexSchr, CastaKrav, FundBicomplex, BCTransmutation, BCBergman, BCAtomic, BCHoiv, BCHarmVek, BCSchwarz, BCBeltrami} among many others. We consider a version of \eqref{eq: dbvpintro} where solutions are functions of a single complex variable that take values in the bicomplex numbers. Using the idempotent representation of bicomplex numbers (which extends to bicomplex-valued functions), we show that classical results for \eqref{eq: dbvpintro} also hold for its bicomplex version, as well as the recent result from \cite{Dirichletbianalytic} on well-posedness by reducing to the complex case. With the same methods, we go on to show that the bicomplex version of the Dirichlet problem for second-order iterated Vekua equations also recovers the known results for the associated complex problem from \cite{WBD2} as well as the new theorem concerning well-posedness. This leads to a complete characterization of solvability of the Dirichlet problem for arbitrarily higher-order Vekua equations in terms of solvability of two complex Dirichlet problems or associated integral conditions.

    We describe the layout of the article. Section \ref{section: background} includes the necessary background for the variety of objects that we consider throughout the paper. Also, we prove the Dirichlet problem for a second-order Vekua equation is well-posed on the same sets as the Dirichlet problem for the Bitsadze equation with a related boundary condition. In Section \ref{section: BicomplexExtension}, we extend many results from the literature for the complex Dirichlet problem into the setting of bicomplex-valued functions. In particular, we recover ill-posedness on the unit circle for all orders, well-posedness on non-degenerate conics that are not circumferences for the bicomplex variant of the Bitsadze equation and the second-order Vekua equation, and characterize solvability on the disk for arbitrary order.

\section{Background}\label{section: background}
    Let $\disk$ be the complex unit disk. For a set $S \subset \mathbb{C}$, let $L^p(S)$ be the Lebesgue space of complex-valued functions on $S$ with integrable modulus raised to the $p^{\text{th}}$ power, $W^{k,p}(S)$ be the Sobolev space of complex-valued functions on a set $S$ with derivatives up to order $k$ which are all in $L^p(S)$, $C(S)$ be the space of complex-valued functions that are continuous on $S$, $C^{0,\alpha}(S)$ be the space of complex-valued $\alpha$-H\"older continuous functions on $S$, and $Hol(S)$ be the collection of complex-valued holomorphic functions on $S$. 

    \subsection{Generalized Analytic Functions}

        \subsubsection{The Classical Theory}

            Complex partial differential equations of the form 
            \begin{equation}\label{eq: vekintroeqn}
                \frac{\p w}{\p\z} = Aw + B\overline{w},
            \end{equation}
            where $\frac{\p}{\p\z} := \frac{1}{2} \left( \frac{\p}{\p x} + i \frac{\p}{\p y} \right)$ is the first-order Cauchy-Riemann differential operator and $A$ and $B$ are taken to be in a Lebesgue space $L^q$, $q>2$, are called Vekua equations. The solutions of a Vekua equation are called generalized analytic functions \cite{Vek} or pseudoanalytic functions \cite{Bers}. Generalized analytic functions are a well-studied class of functions as the Vekua equation is the canonical form of a first-order elliptic equation in the plane and generalized analytic functions recover many properties of the classical holomorphic functions, i.e., solutions $w$ of the Cauchy-Riemann equation $\frac{\p w}{\p\z} = 0$.

            Every solution of a Vekua equation \eqref{eq: vekintroeqn} can be represented as a product of an exponential function that depends on the parameters $A$ and $B$ and a holomorphic function. This representation is called the ``similarity principle'' and is described in the following theorem. 

            \begin{theorem}[``The Basic Lemma'' \cite{Vek}]\label{simprin}
            Let $A, B \in L^q(\disk)$, $q>2$. Every $w: \disk\to\mathbb{C}$ that solves
            \[
                \frac{\p w}{\p\z} = Aw + B\overline{w}
            \]
            has a representation
            \begin{equation}\label{eq: simprinrep}
            w = e^\phi h,
            \end{equation}
            where $e^\phi \in C^{0,\alpha}(\overline{\disk})$ and $h \in Hol(\disk)$.
            \end{theorem}

            The specific form of $\phi$ in the last theorem is 
            \[
                \phi(z) := T_{\disk}\left[A + B\frac{\overline{w}}{w}\right](z),
            \]
            where
            \[
                T_{\disk}[\cdot](z):=-\frac{1}{\pi} \iint_\disk \frac{(\cdot)(\zeta)}{\zeta - z} \,d\xi\,d\eta, \quad\quad \zeta = \xi + i\eta,
            \]
            is the classic Vekua(-Cauchy-Pompeiu-Teodorescu) integral operator.

        \begin{comm}
            For more information on the classic Vekua(-Cauchy-Pompeiu-Teodorescu) integral operator, see, for example, \cite{Vek, BegBook,ellipquasi}.
        \end{comm}

         \subsubsection{Higher-order Vekua Equations}\label{subsubsection: hoiv equations}

            An immediate higher-order generalization of the Cauchy-Riemann equation is to construct a differential equation where the left-hand side of the equation has the Cauchy-Riemann differential operator composed with itself multiple times. In other words, for $n$ a positive integer, consider the equation
            \begin{equation}\label{eq: polyintroeqn}
                \frac{\p^n w}{\p\z^n} = 0.
            \end{equation}
            Solutions of \eqref{eq: polyintroeqn} are called polyanalytic functions and have the form 
            \[
                w = \sum_{k=0}^{n-1} \z^k h_k,
            \]
            where $h_k$ is holomorphic for every $k$. In the special case of $n=2$, \eqref{eq: polyintroeqn} is called the Bitsadze equation, and its solutions are called bianalytic functions. See \cite{Balk}. This idea of iterating a differential operator associated with a first-order complex differential equation is applied to the Vekua equation on $\disk$ in \cite{metahardy, WB3, WBD, WBD2, itvek, itvekbvp} by considering \eqref{eq: vekintroeqn} as 
            \[
            \left( \frac{\p}{\p\z} - A- BC\right)w = 0,
            \]
            where $C$ is the operator that applies complex conjugation, and iterating the operator on the left-hand side. The result is the higher-order iterated Vekua (or HOIV) equation
            \begin{equation}\label{eq: hoivintro}
                \left( \frac{\p}{\p\z} - A- BC\right)^n w = 0,
            \end{equation}
            where $n$ is a positive integer and the functions $A, B$ are taken to be in a Sobolev space $W^{n-1, q}(\disk)$, $q>2$. Solutions of \eqref{eq: hoivintro}, called HOIV functions or generalized polyanalytic functions, are similar in form to the polyanalytic functions except that, instead of being a polynomial in $\z$ with holomorphic coefficients, the functions are polynomials in $z+\z$ with coefficients that are solutions of the associated first-order Vekua equation. The following theorem makes this precise. 

            \begin{theorem}[Theorem 3.1 \cite{WBD}]\label{thm: wbdrep}
                Let $n$ be a positive integer and $A, B \in W^{n-1, q}(\disk)$, $q>2$. A function $w :\disk\to\mathbb{C}$ is a solution of the higher-order iterated Vekua equation 
                \[
                \left( \frac{\p}{\p\z} - A- BC\right)^n w = 0
                \]
                if and only if $w$ has the form 
                \begin{equation}\label{eq: hoivrepintrothm}
                    w(z) = \sum_{k=0}^{n-1} (z+\z)^k \varphi_k(z),
                \end{equation}
                where $\varphi_k$ is a solution of the Vekua equation
                \[
                    \frac{\p \varphi_k}{\p\z} = A \varphi_k + B\overline{\varphi_k},
                \]
                for every $k$.
            \end{theorem}

             Using \eqref{eq: simprinrep}, we can rearrange \eqref{eq: hoivrepintrothm} and see that solutions of 
            \[
                \left( \frac{\p}{\p\z} - A \right)^n w = 0
            \]
            have the form 
            \begin{align*}
                w(z) &= \sum_{k=0}^{n-1} (z+\z)^k \varphi_k(z)\\
                &= \sum_{k=0}^{n-1} (z+\z)^k e^{T_\disk[A](z)} g_k(z)  \\
                &= e^{T_\disk[A](z)} \sum_{k=0}^{n-1} (z+\z)^k  g_k(z) \\
                &= e^{T_\disk[A](z)} \sum_{k=0}^{n-1} \z^k h_k(z),
            \end{align*}
            where $g_k, h_k$ are holomorphic functions, for every $k$. This is made precise with the following corollary of Theorem \ref{thm: wbdrep}. 

            \begin{corr}[Theorem 3.3 \cite{WB3}]\label{corr: hoivrepbzero}
                Let $n$ be a positive integer and $A \in W^{n-1, q}(\disk)$, $q>2$. Every solution of 
                \[
                     \left( \frac{\p}{\p\z} - A \right)^n w = 0
                \]  
                has the form 
                \[
                    w(z) = e^{T_\disk[A](z)}\sum_{k=0}^{n-1} \z^k h_k(z),
                \]
                where $h_k \in H(\disk)$, for every $k$. 
            \end{corr}

            By Theorem \ref{simprin}, since every function $f(z) := e^{T_\disk[A](z)}h(z)$, where $h \in Hol(\disk)$, is a solution of $\frac{\p f}{\p\z} = Af$, we also have the following representation. 

            \begin{corr}[Theorem 3.1 \cite{WBD}, Theorem 6.3 \cite{BCHoiv}] \label{cor: altbzerohoivrep}
                Let $n$ be a positive integer and $A \in W^{n-1, q}(\disk)$, $q>2$. Every solution of 
                \[
                     \left( \frac{\p}{\p\z} - A \right)^n w = 0
                \]  
                has the form 
                \[
                    w(z) = \sum_{k=0}^{n-1} \z^k \varphi_k(z),
                \]
                where $\varphi_k$ solves
                \[
                    \frac{\p \varphi_k}{\p\z} = A\varphi_k,
                \] 
                for every $k$. 
            \end{corr}

            Note that the restriction to $\disk$ here is unnecessary, as all the statements above are also true with $\disk$ replaced with a bounded simply connected domain (we will not consider domains in the sequel with boundary less regular than smooth), and an immediate generalization of Corollary \ref{corr: hoivrepbzero} is the following.

            \begin{theorem}\label{thm: hoivrepgendom}
                Let $n$ be a positive integer, $D$ be a bounded simply connected subset of $\mathbb{C}$, and $A \in W^{n-1, q}(D)$, $q>2$. Every solution of 
                \[
                     \left( \frac{\p}{\p\z} - A \right)^n w = 0
                \]  
                has the form 
                \[
                    w(z) = e^{T_D[A](z)}\sum_{k=0}^{n-1} \z^k h_k(z),
                \]
                where $h_k \in Hol(D)$, for every $k$, and 
                \[
                    T_D[A](z) = -\frac{1}{\pi} \iint_D \frac{A(\zeta)}{\zeta - z} \,d\xi\,d\eta, \quad\quad \zeta = \xi + i\eta.
                \]         
            \end{theorem}

    \subsection{Dirichlet Problem}

        \subsubsection{Harmonic Functions}

            The Dirichlet problem for harmonic functions is
            \begin{equation}\label{eq: classicaldirichlet}
                \begin{cases}
                    \Delta u := 4 \frac{\p^2 w}{\p z \p\z} = 0,  & \text{ in } D,\\
                    u = f, & \text{ on } \p D,
                \end{cases}
            \end{equation}
            for $D \subset \mathbb{C}$ a bounded simply connected domain with smooth boundary and $f$ a nice enough function (or a distribution). This Dirichlet problem is a well-defined boundary value problem in the plane. This means that, for a choice of $f$, there is a \textit{unique} solution to the problem. In particular, when $f \equiv 0$, the problem is only trivially solvable. See, for example, \cite{GK,BegBook, Straube}.

        \subsubsection{Bianalytic Functions}\label{sec: 2.2.2}

            In contrast, it is known that if the Laplace operator $\Delta := 4 \frac{\p^2 }{\p z \p\z}$ is replaced with the Bitsadze operator $\frac{\p^2 }{\p\z^2}$, i.e., the differential operator associated with second-order polyanalytic (called bianalytic) functions, then the resulting Dirichlet problem 
            \[
                \begin{cases}
                    \frac{\p^2 w}{\p\z^2} = 0,  & \text{ in } D,\\
                    w = f, & \text{ on } \p D,
                \end{cases}
            \]
            is not well-defined when $D = \disk$. This is the content of the following proposition attributed to Bitsadze \cite{Bitsadze1948}. We recall the statement as it appears in \cite{Beg04}. 

            \begin{prop}[Lemma 2 \cite{Beg04}]\label{prop: dirichletbianalyticcircle}
                The Dirichlet problem for the Bitsadze equation
                \begin{equation}\label{eq: classicBitsadzeDirichlet}
                    \begin{cases}
                    \frac{\p^2 w}{\p\z^2} = 0,  & \text{ in } \disk,\\
                    w = 0, & \text{ on } \p \disk,
                \end{cases}
                \end{equation}
                has infinitely many solutions.
            \end{prop}

            \begin{comm}The collection of functions $\{w_k\}_{k=0}^\infty$, where 
            \[
                w_k(z) = (1-z\z)z^k,
            \]
            for each $k$, is one example of an infinite collection of solutions to the problem \eqref{eq: classicBitsadzeDirichlet}. \end{comm}

            This pathological behavior is surprising as the unit disk $\disk$ and its boundary the unit circle $\p\disk$ are often the model domain and corresponding boundary for desirable phenomena in the complex plane. Recently, Moreno Garc\'ia, et al., in \cite{Dirichletbianalytic}, show the bianalytic Dirichlet problem \textit{is} well-defined when the boundary set is a nondegenerate conic 
            \[
                Q(x,y) = ax^2 + bxy + cy^2 + dx + ey + f = 0,
            \]
            where $a,b,c,d,e,f \in \mathbb{R}$, that is not a circumference and the boundary function is a polynomial. This means the cases $a=c$ and $b = 0$ are excluded. We recall the relevant theorem from \cite{Dirichletbianalytic} below.

            \begin{theorem}[Theorem 3.1 \cite{Dirichletbianalytic}]\label{thm: dirichletbianalyticwellposed}
                Let $Q(x,y) = 0$ be a non-degenerate conic that is not a circumference, $\Gamma = \{(x,y) \in \mathbb{R}^2 : Q(x,y) = 0\}$, and $P$ be a polynomial in $x$ and $y$. The Dirichlet problem 
                \[
                    \begin{cases}
                        \frac{\p^2 w}{\p\z^2}  = 0, & \text{ in } \mathbb{R}^2\setminus \Gamma,\\
                        w = P, & \text{ on } \Gamma,
                    \end{cases}
                \]
                has a unique solution.
            \end{theorem}

            \begin{comm} The exclusion of circumferences is strict, as the case of $\p\disk$ ($a = c = 1, b = d = e = 0$, and $f = -1$) in Theorem \ref{prop: dirichletbianalyticcircle} illustrates. 
            \end{comm}

            In \cite{WBD2}, the author with B. B. Delgado showed that Proposition \ref{prop: dirichletbianalyticcircle} generalizes to the Dirichlet problem for the second-order iterated Vekua equation 
            \begin{equation}\label{eq: hoivbzero}
                \left( \frac{\p}{\p\z} - A\right)^2 w = 0.
            \end{equation}
            Since \eqref{eq: hoivbzero} reduces to the Bitsadze equation when $A \equiv 0$, it was called a generalized Bitsadze equation in \cite{WBD2}. We refer to \eqref{eq: hoivbzero} as a Vekua-Bitsadze equation, and its solutions as generalized bianalytic functions. The generalization of Proposition \ref{prop: dirichletbianalyticcircle} from \cite{WBD2} is the following proposition.

            \begin{prop}[Lemma 4.2 \cite{WBD2}]\label{lem: generalbitsadzelindep}
    For $A \in W^{1,q}(\disk)$, $q>2$, the Dirichlet problem for the Vekua-Bitsadze equation
    \begin{equation}\label{dbvp}\begin{cases}
        \left( \frac{\p}{\p\z} - A\right)^2 w = 0, & \text{ in } \disk\\
        w = 0, & \text{ on } \p \disk,
    \end{cases}\end{equation}
    has infinitely many linearly independent solutions.
    \end{prop}

        \begin{comm}
            While Theorem \ref{thm: dirichletbianalyticwellposed} guarantees solutions to the associated Dirichlet problem that are defined on the entire plane, we focus in this work on solutions that are defined on the region enclosed by the specified boundary. This avoids technical considerations concerning the coefficient function $A$ on unbounded domains in the results that follow. 
        \end{comm}

            The similarity between the Dirichlet problem for the Bitsadze equation and the Dirichlet problem for the Vekua-Bitsadze equation indicated by the conclusions of Proposition \ref{prop: dirichletbianalyticcircle} and Proposition \ref{lem: generalbitsadzelindep} provides inspiration that the conclusion of Theorem \ref{thm: dirichletbianalyticwellposed} also holds for the Dirichlet problem for the Vekua-Bitsadze equation. However, the boundary functions are not of the same class. Now, the boundary function is a product of a polynomial and an exponential function that is determined by the specific Vekua-Bitsadze equation. This generalizes Theorem 3.1 from \cite{Dirichletbianalytic} to the Vekua-Bitsadze equation, and we prove the theorem below. 

    \begin{theorem}\label{thm: mainthm}
        Let $Q(x,y) = 0$ be a non-degenerate conic that is not a circumference, $\Gamma = \{(x,y) \in \mathbb{R}^2 : Q(x,y) = 0\}$, $D$ be the region enclosed by $\Gamma$, $P$ be a polynomial in $x$ and $y$, and $A \in W^{1,q}(D)$, $q>2$. The Dirichlet problem 
        \begin{equation}\label{eq: complexbitsadzedirichlet}
            \begin{cases}
                \left( \frac{\p}{\p\z} - A \right)^2 w = 0, & \text{ in } D,\\
                w = e^{T_D[A]}P, & \text{ on } \Gamma,
            \end{cases}
        \end{equation}
        has a unique solution. 
    \end{theorem}

    \begin{proof}
        By Theorem \ref{thm: hoivrepgendom}, every solution of 
        \[
            \left( \frac{\p}{\p\z} - A \right)^2 w = 0
        \]
        has the form 
        \[
            w(z) = e^{T_D[A]}(h_0(z) + \z h_1(z)),
        \]
        where $h_0,h_1 \in Hol(D)$. By Theorem \ref{thm: dirichletbianalyticwellposed}, there exist unique $h_0, h_1$ such that 
        \[
            b(z) := h_0(z) + \z h_1(z)
        \]
        is the unique solution of the bianalytic Dirichlet problem
        \begin{equation}\label{bianalyticintermediaryDirichlet}
            \begin{cases}
                \frac{\p^2 b}{\p\z^2}= 0, & \text{ in } D,\\
                b = P, & \text{ on } \Gamma.
            \end{cases}
        \end{equation}
        Using this $h_0,h_1$, 
        \[
            w(z) = e^{T_D[A](z)}(h_0(z) + \z h_1(z))
        \]
        solves 
        \[
            \begin{cases}
                \left( \frac{\p}{\p\z} - A \right)^2 w = 0, & \text{ in } D,\\
                w = e^{T_D[A]}P, & \text{ on } \Gamma.
            \end{cases}
        \]
        Uniqueness follows from the uniqueness of the solution of \eqref{bianalyticintermediaryDirichlet}. 
        
    \end{proof}

    \begin{comm}\label{remark: hoivdbvp}
        The argument used in the proof of Theorem \ref{thm: mainthm} directly extends to the Dirichlet problem for HOIV equations of order greater than two if a higher-order extension of Theorem \ref{thm: dirichletbianalyticwellposed} is available.
    \end{comm}

    \begin{comm}\label{remark: complexremark}
        In \cite{Dirichletbianalytic}, an algorithmic method for finding the explicit solution $b$ of \eqref{bianalyticintermediaryDirichlet} in the proof above is given for certain conics $Q$ and polynomials $P$. Since solutions to \eqref{eq: complexbitsadzedirichlet} all have the form $w = e^{T_D[A]} b$, then these techniques can be applied to this problem to find explicit solutions when $A$ is such that $T_D[A]$ is explicitly computable.
    \end{comm}

            \subsubsection{Higher-Order}

            While we emphasize Dirichlet problems involving second-order differential equations, as this is the natural generalization of the classical case \eqref{eq: classicaldirichlet}, many of the known results extend to their associated higher-order equations by similar arguments. We recall some higher-order extensions of the previously considered results in the literature. The next two propositions extend Proposition \ref{prop: dirichletbianalyticcircle} and Proposition \ref{lem: generalbitsadzelindep}, respectively, to arbitrary order.

            \begin{prop}[Lemma 3.1 \cite{Karaca20}]\label{prop: higherordercomplexdirichlet}
                 Let $n$ be a positive integer. The Dirichlet problem 
                \[
                    \begin{cases}
                    \frac{\p^n w}{\p\z^n} = 0,  & \text{ in } \disk,\\
                    w = 0, & \text{ on } \p \disk,
                \end{cases}
                \]
                has infinitely many linearly independent solutions.
            \end{prop}

        \begin{prop}[Lemma 4.4 \cite{WBD2}]\label{prop: generalbitsadzelindepextension}
                Let $n$ be a positive integer. For $A \in W^{n-1,q}(\disk)$, $q>2$, the Dirichlet problem for the higher-order iterated Vekua equation
                \begin{equation}
                    \begin{cases}
                    \left( \frac{\p}{\p\z} - A\right)^n w = 0, & \text{ in } \disk\\
                    w = 0, & \text{ on } \p \disk,
                    \end{cases}
                \end{equation}
                has infinitely many linearly independent solutions.
        \end{prop}

            The next two theorems provide necessary and sufficient conditions so that the associated Dirichlet problems on the disk are uniquely solvable. These are used in Section \ref{section: BicomplexExtension} and are included here for convenient reference. 

                \begin{theorem}[Theorem 15 \cite{BegBVPCA2}, Theorem 3.2 \cite{Karaca20}]\label{thm: complexhigherorderdirichletdisk}
                    Let $n$ be a positive integer, $f \in L^1(\disk)$, and $\gamma_k \in C(\p\disk)$. The Dirichlet problem
                    \[
                        \begin{cases}
                            \frac{\p^n w}{\p \z^n} = f, &\text{ in } \disk,\\
                            \frac{\p^k w}{\p \z^k} = \gamma_k, & \text{ on } \p\disk,
                        \end{cases}
                    \]
                    for $0 \leq k \leq n-1$, is uniquely solvable if and only if, for $z \in \disk$ and $0 \leq k \leq n-1$, 
                    \[
                        \sum_{\lambda = k}^{n-1} \frac{\z}{2\pi i} \int_{|\zeta|=1} (-1)^{\lambda - k}\frac{\gamma_{\lambda}(\zeta)}{1-\z\zeta} \frac{\overline{(\zeta - z)}^{\lambda-k}}{(\lambda -k)!}\,d\zeta + \frac{(-1)^{n-\lambda}\z}{\pi} \iint_{|\zeta|<1}\frac{f(\zeta)}{1-\z\zeta} \frac{\overline{(\zeta-z)}^{\,n-1-\lambda}}{(n-1-\lambda)!}\,d\xi\,d\eta = 0,
                    \]
                    where $\zeta = \xi + i\eta$.
                    The solution is 
                    \begin{align*}
                        w(z) &= \sum_{\lambda = 0}^{n-1} \frac{(-1)^\lambda}{2\pi i} \int_{|\zeta|=1} \frac{\gamma_\lambda(\zeta)}{\lambda!} \frac{\overline{(\zeta - z)}^\lambda}{\zeta - z}\,d\zeta + \frac{(-1)^n}{\pi} \iint_{|\zeta|<1} \frac{f(\zeta)}{(n-1)!} \frac{\overline{(\zeta -z)}^{\,n-1}}{\zeta -z}\,d\xi\,d\eta.
                    \end{align*}
                \end{theorem}

        \begin{theorem}[Theorem 4.3, Theorem 4.5 \cite{WBD2}]\label{thm: thmfourpointfivewbd2}
            Let $n$ be a positive integer, $A \in W^{n-1,q}(\disk)$, $q>2$, and $\gamma_k \in C(\p\disk)$, $0 \leq k \leq n-1$. The Dirichlet problem
            \[
                \begin{cases}
                    \left( \frac{\p}{\p\z}-A \right)^n w = 0, & \text{ in } \disk,\\
                    \left( \frac{\p}{\p\z}-A \right)^k w = \gamma_k, & \text{ on } \p\disk,\\
                \end{cases}
            \]
            for $0 \leq k \leq n-1$, is uniquely solvable if and only if, for $|z| < 1$  and $0 \leq k \leq n-1$,
            \[
                \sum_{\lambda = k}^{n-1}\frac{\z}{2\pi i} \int_{|\zeta| = 1} (-1)^{\lambda - k} \frac{\gamma_\lambda(\zeta)e^{-T_\disk[A](\zeta)}}{1- \z\zeta} \frac{(\overline{\zeta - z})^{\lambda- k}}{(\lambda - k)!}\,d\zeta   = 0.
            \]

            The solution is 
            \[
                w = e^{T_\disk[A]}  \varphi, 
            \]
            where
            \[
                \varphi(z) =  \sum_{k=0}^{n-1} \frac{(-1)^k}{2\pi i} \int_{|\zeta| = 1} \frac{\gamma_k(\zeta)e^{-T_\disk[A](\zeta)}}{k!} \frac{(\overline{\zeta - z})^k}{\zeta - z}\,d\zeta .
            \]
            
         \end{theorem}

    \subsection{Bicomplex Numbers}

        In Section \ref{section: BicomplexExtension}, we prove generalizations of Proposition \ref{prop: dirichletbianalyticcircle}, Theorem \ref{thm: dirichletbianalyticwellposed}, Proposition \ref{lem: generalbitsadzelindep}, Proposition \ref{prop: higherordercomplexdirichlet}, Proposition \ref{prop: generalbitsadzelindepextension}, Theorem \ref{thm: complexhigherorderdirichletdisk}, and Theorem \ref{thm: thmfourpointfivewbd2} in the setting of bicomplex-valued functions of a single complex variable. In the remainder of this section, we provide necessary background information on the bicomplex numbers and associated results that are used in Section \ref{section: BicomplexExtension}. For the reader that is only interested in content concerning complex-valued functions, the remainder of this section and Section \ref{section: BicomplexExtension} may be ignored.

        \subsubsection{Basics}

             The bicomplex numbers, denoted by $\bc$, are the set $\mathbb{C}^2$ with the operations of component-wise addition and multiplication defined by 
            \[  
                (z_1, z_2) \cdot (w_1,w_2) = (z_1w_1-z_2w_2, z_1w_2 + z_2w_1),
            \]
            for all $(z_1, z_2), (w_1, w_2) \in \bc$. It is convenient to realize $(z_1, z_2) \in \bc$ as $z_1 + jz_2$, where $j^2 = -1$. With this representation, the usual algebraic operations of the complex numbers agree with the operations of the bicomplex numbers as defined above. See, for example, \cite{PriceMulti, BCHolo, ComplexSchr, CastaKrav, BCTransmutation, BCBergman, FundBicomplex} for extensive background on the bicomplex numbers and their role in analysis.
            Next, let us recall the basic terminology of the bicomplex numbers. 

            \begin{deff}
                For $z = z_1 + jz_2 \in \bc$, $z_1$ is the scalar part of $z$, denoted by $\sca z$, and $z_2$ is the vector part of $z$, denoted by $\vect z$. The bicomplex conjugate of $z\in \bc$ is 
                \[
                    \overline{z} = \sca z - j \vect z.
                \]
            \end{deff}

            \begin{comm}
                To differentiate between the bicomplex conjugate of $w \in \bc$ and the complex conjugate of $z \in \mathbb{C}$, throughout the remainder of this section and in Section \ref{section: BicomplexExtension}, we adopt the convention that the complex conjugate of $z = x+ iy\in \mathbb{C}$ is denoted by $z^* = x-iy$. 
            \end{comm}

            \begin{deff}
                For $z = x + i y \in \C$, i.e., $x, y \in \R$, the bicomplexification of $z$ is 
                \[
                    \widehat{z} = x + jy.
                \]  
            \end{deff}

            The bicomplex numbers have two idempotent elements
            \[
            p^+ := \frac{1}{2}(1 + ji) \quad \text{ and } \quad  p^- := \frac{1}{2}(1 - ji).
            \]
            Observe that 
            \[
                (p^\pm)^2 = p^\pm, \quad p^+ + p^- = 1, \text{ and } \quad p^+p^- = 0.
            \]
            The next proposition describes that every bicomplex number can be written as a linear combination of these two numbers. 

            \begin{prop}[Proposition 1 \cite{BCTransmutation}]\label{everybchasplusandminus}
                Let $w \in \bc$. There exist unique $w^{\pm} \in \C$ such that 
                \[
                    w = p^+ w^+ + p^- w^-.
                \]
                Furthermore, 
                \[
                    w^{\pm} = \sca w \mp i \vect w.
                \]
                \end{prop}

            Since we use a bicomplex-valued exponential function in Theorem \ref{thm: bcmainthm}, we explicitly describe this object here, as it is found in \cite{FundBicomplex}. See also \cite{BCHolo}.

            \begin{deff}
                For $w = p^+ w^+ + p^- w^- \in \bc$, the bicomplex exponential function $e^w$ is defined to be 
                \[
                    e^w := p^+ e^{w^+} + p^- e^{w^-}.
                \]
            \end{deff}

            Next, we define the bicomplex norm and the Lebesgue and Sobolev spaces of bicomplex-valued functions of a single complex variable. We follow this with a result that relates a function's inclusion in the bicomplex-Lebesgue space with the inclusion of the component functions in its idempotent representation in a Lebesgue space. 

            \begin{deff}
                For $w \in \bc$, the bicomplex norm of $w$, denoted by $||\cdot||_{\bc}$,  is
                \[
                    ||w||_{\bc} := \sqrt{\frac{|w^+|^2 + |w^-|^2}{2}},
                \]
                where $|w^{\pm}|$ is the complex modulus of $w^\pm$. For $S \subset \mathbb{C}$, $p$ a positive real number, and $k$ a nonnegative integer, we denote by $C(S, \bc)$ the collection of functions $w:S \to\bc$ that are continuous with respect to the metric induced by the bicomplex norm, we define the bicomplex Lebesgue space $L^p(S,\bc)$ as the collection of functions $f: \disk\to\bc$ such that
            \[
                ||f||_{L^p(S,\bc)} := \left( \iint_{\disk}||f(z)||^p_\bc\,dx\,dy\right)^{1/p} < \infty,
            \]  
            and we define the bicomplex Sobolev space $W^{k,p}(S,\bc)$ as the collection of functions $f: S\to\bc$ such that $f$ and its derivatives up to order $k$ are in $L^p(S,\bc)$. 
            \end{deff}

            \begin{prop}[Section 2.2 \cite{BCBergman}]\label{cor: corollaryof2.24BCAtomic}\label{prop: bcctsiffidemc}
                Let $D \subset \mathbb{C}$ be a bounded domain. For $p$ a positive real number, a function $f = p^+ f^+ + p^- f^- \in L^p(D, \bc)$ if and only if $f^\pm \in L^p(D)$, and a function $f = p^+ f^+ + p^- f^- \in C(D, \bc)$ if and only if $f^\pm \in C(D)$.
            \end{prop}

         \subsubsection{Holomorphic and Generalized Analytic Functions}

            Next, we define the bicomplex differential operators and related concepts that we consider in the sequel. 

            \begin{deff}\label{bcdbardef}
                The bicomplex differential operators $\p$ and $\dbar$ are
                \[
                    \p := \frac{1}{2} \left( \frac{\p}{\p x} - j \frac{\p }{\p y}\right)
                \]
                and
                \[
                    \dbar := \frac{1}{2} \left( \frac{\p}{\p x} + j \frac{\p }{\p y}\right).
                \]
                Alternatively, the differential operators can be represented using the idempotent elements $p^\pm$ as 
                \[
                    \p = p^+ \frac{\p}{\p z^*} + p^- \frac{\p }{\p z}
                \]
                and
                \[
                    \dbar = p^+ \frac{\p }{\p z} + p^- \frac{\p}{\p z^*},
                \]
                where $\frac{\p}{\p z} := \frac{1}{2}\left(\frac{\p}{\p x} -i \frac{\p}{\p y}\right)$ and $\frac{\p}{\p z^*} := \frac{1}{2}\left(\frac{\p}{\p x} + i \frac{\p}{\p y}\right)$ are the usual first-order complex differential operators.
            \end{deff}

            \begin{comm}
                Observe that, for $z$ a complex variable and $n$ a positive integer, we have
                \begin{align*}
                    \p \widehat{z^*} &= 0 \\
                    \dbar \widehat{z} &=0 \\
                    \p \widehat{z}^{\,n} &= n \widehat{z}^{\,n-1}\\
                    \dbar \widehat{z^*}^{\,n} &= n \widehat{z^*}^{\,n-1}.
                \end{align*}
            \end{comm}

            \begin{deff}
            The $\bc$-holomorphic functions on a set $S \subset\mathbb{C}$ are the collection of functions $w: S \to\bc$ that satisfy
            \begin{equation}\label{eqn: bccreqn}
                \dbar w = 0.
            \end{equation}
            We indicate the collection of these functions by $Hol(S,\bc)$. 
            \end{deff}

            \begin{comm}
                While we only consider bicomplex-valued functions of a single complex variable in this work, bicomplex-valued functions of a bicomplex variable and holomorphicity in that context have previously been considered. See \cite{PriceMulti, BCHolo} for more information. 
            \end{comm}

                The next theorem illustrates a useful relationship between $\bc$-holomorphic functions and the complex components of their idempotent representation. This relationship informs much of the work that follows.  

            \begin{theorem}[Remark 2.3 \cite{BCAtomic} (see also Remark 4 of \cite{BCTransmutation} or \cite{BCBergman})]\label{thm: remark3BCAtomic}
                A function $w: \disk \to\bc$ with idempotent representation $w(z) = p^+ w^+(z) + p^- w^-(z)$ satisfies
                \[
                    \dbar w = 0
                \]
                if and only if 
                \[
                    \frac{\p (w^+)^*}{\p z^*} = 0 = \frac{\p w^-}{\p z^*}.
                \]
            \end{theorem}

            Now, we define the analogue of the Vekua equation \eqref{eq: vekintroeqn} in the bicomplex setting.

            \begin{deff}
                Let $A, B \in L^q(\disk,\bc)$, $q>2$. Equations of the form
                \begin{equation}\label{eq: bcvek}
                    \dbar w = Aw + B\overline{w}
                \end{equation}
                are called $\bc$-Vekua equations. Solutions of \eqref{eq: bcvek} are called $\bc$-generalized analytic functions. 
            \end{deff}

            \begin{comm}
                As in the $\bc$-holomorphic case, we only consider bicomplex-valued functions of a single complex variable in this work.  Bicomplex-valued generalized analytic functions of a bicomplex variable are considered in, for example, \cite{ BerglezBCPseudo, BerglezBCPseudo2}. 
            \end{comm}

            The next theorem generalizes Theorem \ref{thm: remark3BCAtomic} to the $\bc$-Vekua equation. 

            \begin{theorem}[Theorem 4.9 \cite{BCHoiv}]\label{thm: 4.9BCHoiv}
                Let $A, B \in L^q(\disk,\bc)$, $q>2$. A function $w = p^+ w^+ + p^- w^-: \disk\to\bc$ solves
                \[
                    \dbar w = Aw + B\overline{w}
                \]
                if and only if 
                \[
                    \frac{\p (w^+)^*}{\p z^*} = (A^+)^*(w^+)^* + (B^+)^* w^+
                \]
                and
                \[
                    \frac{\p w^-}{\p z^*} = A^- w^- + B^- (w^-)^*
                \]
            \end{theorem}

            Theorem \ref{thm: remark3BCAtomic} and Theorem \ref{thm: 4.9BCHoiv} make clear that properties of complex holomorphic and generalized analytic functions that are maintained in linear combination (after one component function has been conjugated) will be realized by their bicomplex versions. See \cite{BCAtomic, BCHoiv, BCHarmVek, BCBeltrami} for many examples.

         \subsubsection{Solutions to Higher-Order Differential Equations}

            Similarly to the complex case, the bicomplex differential equations considered in the last subsection can be extended to higher-order equations by iterating the operator. We consider these extensions. 

            \begin{deff}
                Let $n$ be a positive integer. The $\bc$-higher-order homogeneous Cauchy-Riemann equation is
                \begin{equation}\label{eq: hobccreqn}
                    \dbar^n w = 0.
                \end{equation}
                Solutions $w:\disk\to\bc$ of \eqref{eq: hobccreqn} are called $\bc$-polyanalytic functions.    
            \end{deff}

            \begin{deff}
                Let $n$ be a positive integer and $A, B \in W^{n-1, q}(\disk,\bc)$. The $\bc$-higher-order iterated Vekua (or $\bc$-HOIV) equations are 
                \begin{equation}\label{eq: bchoiveqn}
                    (\dbar - A - BC)^n w = 0,
                \end{equation}
                where $C$ is the operator that applies bicomplex conjugation. Solutions of \eqref{eq: bchoiveqn} are called $\bc$-HOIV functions or $\bc$-generalized polyanalytic functions.   
            \end{deff}

            The next theorem provides a representation formula for solutions of \eqref{eq: hobccreqn} and \eqref{eq: bchoiveqn} that is similar in structure to those for complex polyanalytic and HOIV functions found in Subsection \ref{subsubsection: hoiv equations}.

            \begin{theorem}[Theorem 6.3 \cite{BCHoiv}]\label{thm: 6point3BCHoiv}
                Let $n$ be a positive integer and $A \in W^{n-1, q}(\disk,\bc)$, $q>2$. Every solution $w: \disk \to\bc$ of 
                \[
                    (\dbar- A)^n w = 0
                \]
                is representable as 
                \[
                    f(z) = \sum_{k=0}^{n-1} \widehat{z^*}^k \varphi_k(z),
                \]
                where $\varphi_k:\disk\to\bc$ satisfies
                \[
                    \dbar \varphi_k = A\varphi_k,
                \]
                for every $k$. 
            \end{theorem}

            While Theorem \ref{simprin} does not in general extend to solutions of $\bc$-Vekua equations, it does in the case that $B\equiv 0 $. We recall an integral operator that is a right-inverse to $\dbar$ and the relevant results. Using this result, we prove a bicomplex extension of Corollary \ref{corr: hoivrepbzero} as a corollary of Theorem \ref{thm: 6point3BCHoiv}. 

            \begin{deff}
                Let $D$ be a bounded simply connected domain of $\mathbb{C}$. For $f \in L^p(D, \bc)$, $p\geq 1$, the bicomplex Vekua(-Cauchy-Pompeiu-Teodorescu) operator $T_D^\bc(\cdot)$ defined by 
                \[
                    T_D^\bc[f](z) := p^+\left( -\frac{1}{\pi} \iint_D \frac{f^+(\zeta)}{\zeta^*-z^*}\,d\xi\,d\eta\right) + p^- \left(-\frac{1}{\pi} \iint_D \frac{f^-(\zeta)}{\zeta-z}\,d\xi\,d\eta\right),\quad \zeta = \xi + i\eta,
                \]
                exists.
            \end{deff}

            \begin{comm}
                By their respective definitions, observe that
                \begin{align*}
                    T_D^\bc[f](z) 
                    &= p^+ (T_D[(f^+)^*])^* + p^- T_D[f^-].
                \end{align*}
                We encounter this representation again in the first proof of Theorem \ref{thm: bcmainthm}.
            \end{comm}

            \begin{theorem}[Section 3.1 \cite{FundBicomplex}, Section 3.2 \cite{BCTransmutation}]
                For $f \in L^p(\disk, \bc)$, $p\geq 1$, the bicomplex Vekua(-Cauchy-Pompeiu-Teodorescu) operator satisfies
                \[
                    \dbar T_\disk^\bc[f] = f.
                \]
            \end{theorem}

            \begin{theorem}[Theorem 4.5 \cite{BCHoiv}]\label{thm: fourpointfivebchoiv}
                Let $A \in L^1(\disk,\bc)$. Every solution $w:\disk\to\bc$ of 
                \[
                    \dbar w = Aw
                \]
                is representable as 
                \[
                    w = e^\phi h
                \]
                where $h \in Hol(\disk,\bc)$ and $\phi = T_\disk^\bc[A]$. In the case that $A \in L^q(\disk,\bc)$, $q>2$, $T_\disk^\bc[A] \in C^{0,\alpha}(\overline{D})$, for $\alpha = \frac{q-2}{q}$. 
            \end{theorem}

            \begin{corr}\label{corr: bcsimprinhoivrep}
                Let $n$ be a positive integer and $A \in W^{n-1, q}(\disk,\bc)$, $q>2$. Every solution $w: \disk \to\bc$ of 
                \[
                    (\dbar- A)^n w = 0
                \]
                is representable as 
                \[
                    w(z) = e^{T_\disk^\bc[A](z)}\sum_{k=0}^{n-1} \widehat{z^*}^k h_k(z),
                \]
                where $h_k\in Hol(\disk,\bc)$, for every $k$. 
            \end{corr}

            \begin{proof}
                By Theorem \ref{thm: 6point3BCHoiv}, 
                \[
                    w(z) = \sum_{k=0}^{n-1} \widehat{z^*}^k \varphi_k(z),
                \]
                where $\varphi_k:\disk\to\bc$ satisfies
                \[
                    \dbar \varphi_k = A\varphi_k,
                \]
                for every $k$. By Theorem \ref{thm: fourpointfivebchoiv}, there exists $h_k \in Hol(\disk,\bc)$ such that 
                \[
                    \varphi_k = e^{T_\disk^\bc[A]} h_k,
                \]      
                for every $k$. Thus, 
                \begin{align*}
                    w(z) &= \sum_{k=0}^{n-1} \widehat{z^*}^k \varphi_k(z) \\
                    &= \sum_{k=0}^{n-1} \widehat{z^*}^k e^{T_\disk^\bc[A](z)} h_k(z) \\
                    &= e^{T_\disk^\bc[A](z)} \sum_{k=0}^{n-1} \widehat{z^*}^k h_k(z).
                \end{align*}
            \end{proof}

            As in the case of complex-valued functions, restriction to $\disk$ here is unnecessary, as all the statements above are also true with $\disk$ replaced with a bounded simply connected domain. An immediate generalization of Corollary \ref{corr: bcsimprinhoivrep} that we use in Section \ref{last one} is the following.

            \begin{theorem}\label{thm: bchoivsimprinreptheorem}
                For a bounded simply connected set $D \subset \mathbb{C}$, let $n$ be a positive integer and $A \in W^{n-1, q}(D,\bc)$, $q>2$. Every solution $w: D \to\bc$ of 
                \[
                    (\dbar- A)^n w = 0
                \]
                is representable as 
                \[
                    w(z) = e^{T_D^\bc[A](z)}\sum_{k=0}^{n-1} \widehat{z^*}^k h_k(z),
                \]
                where $h_k\in Hol(D,\bc)$, for every $k$.
            \end{theorem}

         \subsubsection{Dirichlet Problem}

         Previously, the bicomplex Dirichlet problem 
         \begin{equation}\label{bcdir}\begin{cases}
                \p \dbar w = 0, & \text{ in }\disk\\
                w = g, & \text{ on } \p\disk,
         \end{cases}\end{equation}
         was considered in \cite{BCSchwarz}. Note that since 
         \[
            \p \dbar = \left(p^+ \frac{\p}{\p z^*} + p^- \frac{\p }{\p z} \right) \left(p^+ \frac{\p }{\p z} + p^- \frac{\p}{\p z^*}\right) = 2 \frac{\p^2}{\p z\p z^*} = \frac{1}{2}\Delta , 
         \]
         functions that solve $\p \dbar w = 0$ are harmonic functions in the classical sense. The following two theorems were proved in \cite{BCSchwarz} and show that \eqref{bcdir} shares characteristics with the classical Dirichlet problem \eqref{eq: classicaldirichlet}.
         
         \begin{theorem}[Theorem 4.1 \cite{BCSchwarz}]
             The bicomplex Dirichlet problem
            \[
                \begin{cases}
                  \p \dbar f = 0, &\text{ in } \disk,\\
                  f = 0, &\text{ on } \p\disk,
                \end{cases}
            \]
        has only the trivial solution.
        \end{theorem}

        \begin{theorem}[Theorem 4.2 \cite{BCSchwarz}]
            For $g \in L^1(\p \disk, \bc)$, the bicomplex Dirichlet problem
            \[
                \begin{cases}
                    \p \dbar f = 0, & \text{ in } \disk,\\
                    f = g, & \text{ on } \p\disk,
                \end{cases}
            \]
        is uniquely solved by 
        \[
            f = p^+ \frac{1}{2\pi} \int_0^{2\pi} g^+(e^{it})P_r(\theta - t) \,dt  + p^- \frac{1}{2\pi} \int_0^{2\pi} g^-(e^{it})P_r(\theta - t) \,dt,
        \]
        where $P_r(\theta) := \frac{1-r^2}{1-2r\cos(\theta) + r^2}$ is the Poisson kernel on $\disk$. 
        \end{theorem}
         
         The problem \eqref{bcdir} is not directly considered in this work, but we include it to compare with the classical problem and contrast with the results for the Dirichlet problem for the $\bc$-bianalytic functions considered in Section \ref{section: BicomplexExtension}.

\section{Bicomplex Extension}\label{section: BicomplexExtension}

    The extension of Theorem \ref{thm: dirichletbianalyticwellposed} that is Theorem \ref{thm: mainthm} concludes our work generalizing the Dirichlet boundary value problem for the Bitsadze equation. Now, we extend some preexisting results from the literature concerning Dirichlet problems for the Bitsadze and Vekua-Bitsadze equations to the setting of bicomplex analogues of the Bitsadze and Vekua-Bitsadze equations. Note, from here until the end of the paper, we use the notational convention that the complex conjugate of a complex-valued $z = x+iy$, i.e., $x, y \in \mathbb{R}$, is $z^* = x-iy$ and the bicomplex conjugate of a bicomplex valued $w = z_1 + j z_2$, i.e., $z_1,z_2 \in \mathbb{C}$, is $\z = z_1 - j z_2$. 

    The $\bc$-Bitsadze equation is given by
    \[
        \dbar^2 w = 0,
    \]
    and the $\bc$-Vekua-Bitsadze equations that we consider are any equation of the form 
    \[
        (\dbar - A)^2 w = 0,
    \]
    where $A \in W^{1,q}(\disk,\bc)$, $q>2$. Clearly, these are the Bitsadze equation and the Vekua-Bitsadze equation with $\frac{\p}{\p z^*}$ replaced by $\dbar$, as well as letting the coefficient function $A$ and the solution $w$ take values in the bicomplex numbers.

    \subsection{The Dirichlet Problem for the Bicomplex Bitsadze Equation}\label{section: bcdirifchletbitsadzeequation}

        \subsubsection{The Unit Circle}

        We prove the conclusion of Proposition \ref{prop: dirichletbianalyticcircle} for the $\bc$-Bitsadze equation and include an immediate corollary that extends the result to higher-order equations. We begin with a structural lemma that relates $\bc$-polyanalyticity of a function with polyanalyticity of the complex components of its idempotent representation. The lemma generalizes Theorem \ref{thm: remark3BCAtomic} to higher-order equations and is crucial to the proof of Proposition \ref{prop: bicomplexbitsadzedirichlet}, as well as to the proof of Theorem \ref{thm: generalizationOfBianalyticWellPosed} later.

            \begin{lem}\label{lem: bicomplexpolyanalyticidempotentrep}
                Let $n$ be a positive integer. A function $w = p^+ w^+ + p^- w^-$ is a solution to
                \[
                    \dbar^n w = 0
                \]
                if and only if 
                \[
                    \frac{\p^n (w^+)^*}{\p (z^*)^n} = 0
                \]
                and 
                \[
                    \frac{\p^n w^-}{\p (z^*)^n} = 0.
                \]
            \end{lem}

            \begin{proof}
                Suppose $w= p^+ w^+ + p^- w^-$ solves
                \[
                    \dbar^n w = 0.
                \]
                By Theorem \ref{thm: 6point3BCHoiv} with $A \equiv 0$,
                \[
                    w(z) = \sum_{k=0}^{n-1} \widehat{z^*}^k h_k(z),
                \]
                where $h_k \in Hol(\disk,\bc)$, for every $k$. By Theorem \ref{thm: remark3BCAtomic}, $(h_k^+)^*, h_k^- \in Hol(\disk)$, for every $k$. Since $\widehat{z^*} = p^+ z + p^- z^*$, it follows that 
                \begin{align*}
                     w(z) &= \sum_{k=0}^{n-1} \widehat{z^*}^k h_k(z)\\
                     &= \sum_{k=0}^{n-1} (p^+ z + p^- z^*)^k (p^+ h_k^+ + p^- h_k^-) \\
                     &= \sum_{k=0}^{n-1} (p^+ z^k + p^- (z^*)^k) (p^+ h_k^+ + p^- h_k^-) \\
                     &= \sum_{k=0}^{n-1} \left(p^+ z^k h_k^+ + p^- (z^*)^k h_k^-\right) \\
                     &= p^+ \sum_{k=0}^{n-1} z^k h_k^+ + p^- \sum_{k=0}^{n-1} (z^*)^k h_k^-.
                \end{align*}
                Thus, 
                \[
                    w^+ = \sum_{k=0}^{n-1} z^k h_k^+ ,
                \]
                and
                \[
                    w^- = \sum_{k=0}^{n-1} (z^*)^k h_k^-.
                \]
                Therefore, by Corollary \ref{corr: hoivrepbzero} with $A \equiv 0$ (also see \cite{Balk}), $(w^+)^*, w^-$ are polyanalytic functions of order $n$. 

                Now, suppose $f, g : \disk \to \mathbb{C}$ are polyanalytic functions of order $n$. Construct the function $w: \disk \to \bc$ by 
                \[
                    w(z) := p^+ (f(z))^* + p^- g(z).
                \]
                Observe that 
                \begin{align*}
                    \dbar^n w &= \left( p^+ \frac{\p}{\p z} + p^- \frac{\p}{\p z^*}\right)^n (p^+ (f)^* + p^- g)\\
                    &= \left( p^+ \frac{\p^n}{\p z^n} + p^- \frac{\p^n}{\p (z^*)^n}\right) (p^+ (f)^* + p^- g)\\
                    &= p^+ \frac{\p^n (f)^*}{\p z^n} + p^- \frac{\p^n g }{\p (z^*)^n} \\
                    &= p^+ \left(\frac{\p^n f}{\p (z^*)^n}\right)^* + p^- \frac{\p^n g }{\p (z^*)^n} \\
                    &= p^+ (0) + p^- ( 0) = 0. 
                \end{align*}
                Therefore, $w$ solves 
                \[
                    \dbar^n w = 0.
                \]
            \end{proof}

            \begin{prop}\label{prop: bicomplexbitsadzedirichlet}
                The Dirichlet problem for the $\bc$-Bitsadze equation
                \begin{equation}\label{eq: bicomplexdirichleteuation}
                    \begin{cases}
                    \dbar^2 w = 0, & \text{ in } \disk,\\
                    w = 0, & \text{ on } \p\disk,
                    \end{cases}
                \end{equation}
                has infinitely many solutions. 
            \end{prop}

            \begin{proof}

                By Lemma \ref{lem: bicomplexpolyanalyticidempotentrep}, every solution of $\dbar^2 w = 0$ has the form 
                \[
                    w = p^+ w^+ + p^- w^-,
                \]
                where
                \[
                    \frac{\p^2 (w^+)^*}{\p (z^*)^2} = 0 = \frac{\p^2 w^-}{\p (z^*)^2}.
                \]
                By Proposition \ref{prop: dirichletbianalyticcircle}, there exist infinitely many functions $f$ that satisfy 
                \begin{equation}\label{eq: bianalyticDirichletFourTwo}
                    \begin{cases}
                        \frac{\p^2 f}{\p(z^*)^2} = 0, & \text{ in } \disk,\\
                        f = 0, & \text{ on } \p\disk.
                    \end{cases}
                \end{equation}
                Choosing two such functions that solve \eqref{eq: bianalyticDirichletFourTwo}, call them $f_1, f_2$, then the constructed function 
                \[
                    w = p^+ (f_1)^* + p^- f_2
                \]
                solves $\dbar^2 w = 0$ and 
                \begin{align*}
                    w|_{\p\disk} &= p^+ (f_1)^*|_{\p\disk} + p^- f_2|_{\p\disk} \\
                    &= p^+ (0)^* + p^- 0 = 0.
                \end{align*}
                Since there are infinitely many pairs of functions $f_1, f_2$ to choose from, it follows that there are infinitely many constructable solutions $w = p^+ (f_1)^* + p^- f_2$ of \eqref{eq: bicomplexdirichleteuation}.

            \end{proof}

            \begin{corr}\label{cor: bicomplexPolyDirichletNotWellPosed}
                Let $n$ be a positive integer. The Dirichlet problem 
                \[
                    \begin{cases}
                    \dbar^n w = 0, & \text{ in } \disk,\\
                    w = 0, & \text{ on } \p\disk,
                    \end{cases}
                \]
                has infinitely many solutions. 
            \end{corr}

            \begin{proof}
                The argument for this result is the same as the one found in the proof of Proposition \ref{prop: bicomplexbitsadzedirichlet} except an appeal to Proposition \ref{prop: higherordercomplexdirichlet} is made instead of Proposition \ref{prop: dirichletbianalyticcircle}.  
            \end{proof}

             \begin{examp}

            The collection of functions 
            \[
                w_j(z) = \sum_{k=1}^{n-1} \widehat{z^*}^{n-1} \widehat{z}^{\,j} (1-(\widehat{z}\widehat{z^*})^{n-k}),
            \]
            for every positive integer $j$, modeled after those found in the proof of Lemma 3.1 in \cite{Karaca20}, provide an example of the infinitely many functions guaranteed to exist by Corollary \ref{cor: bicomplexPolyDirichletNotWellPosed}.
            
            \end{examp}

            Now, we show that solvability of the Dirichlet problem for the $\bc$-Bitsadze equation on the disk is equivalent to solvability of two Dirichlet problems for the Bitsadze equation. Two immediate corollaries follow: one is the higher-order extension of the proposition and the other relates solvability of the Dirichlet problem for the $\bc$-Bitsadze equation with integral conditions. Note that the second corollary is a bicomplex extension of Theorem \ref{thm: complexhigherorderdirichletdisk} and provides a representation formula for solutions. Then, these three results are condensed into the single Theorem \ref{thm: bcdirichletsolvabletfae}.

            \begin{prop}\label{prop: bcdirichletiffcomplexdirichletfornequalstwo}
                Let $f \in L^1(\disk, \bc)$ and $\gamma_0, \gamma_1 \in C(\p\disk, \bc)$. The Dirichlet problem for the $\bc$-Bitsadze equation
                \[
                    \begin{cases}
                        \dbar^2 w = f, & \text{ in } \disk,\\
                        w = \gamma_0, & \text{ on } \p\disk,\\
                        \dbar w = \gamma_1, & \text{ on } \p\disk,
                    \end{cases}
                \]
              is solvable if and only if the Dirichlet problems for the Bitsadze equation
                \[
                    \begin{cases}
                        \frac{\p^2 (w^+)^*}{\p (z^*)^2} = (f^+)^*, & \text{ in } \disk,\\
                        (w^+)^* = (\gamma_0^+)^*, & \text{ on } \p\disk,\\
                        \frac{\p (w^+)^*}{\p z^*} = (\gamma_1^+)^*, & \text{ on } \p\disk,
                    \end{cases}
                \]
                and
                \[
                    \begin{cases}
                         \frac{\p^2 w^-}{\p (z^*)^2} = f^-, & \text{ in } \disk,\\
                        w^- = \gamma_0^-, & \text{ on } \p\disk,\\
                        \frac{\p w^-}{\p z^*} = \gamma_1^-, & \text{ on } \p\disk,
                    \end{cases}
                \]
                are solved by the complex components $w^+$ and $w^-$ of the idempotent representation of $w = p^+ w^+ + p^- w^-$. 
            \end{prop}

            \begin{proof}
                By Proposition \ref{cor: corollaryof2.24BCAtomic}, $f \in L^1(\disk,\bc)$ if and only if $(f^+)^*, f^- \in L^1(\disk)$, and by Proposition \ref{prop: bcctsiffidemc}, $\gamma_0,\gamma_1 \in C(\p\disk,\bc)$ if and only if $(\gamma_0^+)^*,\gamma_0^-, (\gamma_1)^*, \gamma_1^- \in C(\p\disk)$. Similarly to the conclusion of Lemma \ref{lem: bicomplexpolyanalyticidempotentrep}, since 
                \[
                    \dbar^n = p^+ \frac{\p^n}{\p z^n} + p^- \frac{\p^n}{\p (z^*)^n}
                \]
                as differential operators, for every positive integer $n$, it follows that 
                \begin{align*}
                    \dbar^2 w &= f\\
                    \left(p^+ \frac{\p^2}{\p z^2} + p^- \frac{\p^2}{\p (z^*)^2}\right)(p^+ w^+ + p^- w^-) &= p^+ f^+ + p^- f^-\\
                    p^+ \frac{\p^2 w^+}{\p z^2} + p^- \frac{\p w^-}{\p (z^*)^2} &= p^+ f^+ + p^- f^-.
                \end{align*}
                Associating the components of the above idempotent representations, we have that 
                \[
                    \dbar^2 w = f
                \]
                if and only if
                \[
                    \frac{\p^2 (w^+)^*}{\p (z^*)^2} = (f^+)^*
                \]
                and
                \[
                    \frac{\p^2 w^-}{\p (z^*)^2} = f^-.
                \]
                Since 
                \begin{align*}
                    w|_{\p\disk} &= p^+ w^+|_{\p\disk} + p^- w^-|_{\p\disk} 
                \end{align*}
                and
                \begin{align*}
                    \gamma_0 &= p^+ \gamma_0^+ + p^- \gamma_0^-,
                \end{align*}
                it follows that $w= \gamma_0$ on $\p\disk$ if and only if $(w^+)^*|_{\p\disk} = (\gamma_0^+)^*$ and $w^-|_{\p\disk} = \gamma_0^-$. Similarly, since 
                \[
                    \dbar w = p^+ \frac{\p w^+}{\p z} + p^- \frac{\p w^-}{\p z^*}
                \]
                and 
                \begin{align*}
                    \gamma_1 &= p^+ \gamma_1^+ + p^- \gamma_1^-,
                \end{align*}
                it follows that $\dbar w = \gamma_1$ on $\p\disk$ if and only if $\frac{\p (w^+)^*}{\p z^*}|_{\p\disk} = (\gamma_1^+)^*$ and $\frac{\p w^-}{\p z^*}|_{\p\disk} = \gamma_1^-$. 
            \end{proof}

            \begin{corr}\label{corr: bcdirichletiffcomplexdirichletforanyn}
                 Let $n$ be a positive integer, $f \in L^1(\disk, \bc)$, and $\gamma_k \in C(\p\disk, \bc)$, $0\leq k\leq n-1$. The Dirichlet problem
                \[
                    \begin{cases}
                        \dbar^n w = f, & \text{ in } \disk,\\
                        \dbar^k w = \gamma_k, & \text{ on } \p\disk,
                    \end{cases}
                \]
                for $0 \leq k \leq n-1$, is solvable if and only if the Dirichlet problems
                \[
                    \begin{cases}
                        \frac{\p^n (w^+)^*}{\p (z^*)^n} = (f^+)^*, & \text{ in } \disk,\\
                        \frac{\p^k (w^+)^*}{\p (z^*)^k} = (\gamma_k^+)^*, & \text{ on } \p\disk,
                    \end{cases}
                \]
                for $0 \leq k \leq n-1$, and
                \[
                    \begin{cases}
                         \frac{\p^n w^-}{\p (z^*)^n} = f^-, & \text{ in } \disk,\\
                        \frac{\p^k w^-}{\p (z^*)^k} = \gamma_k^-, & \text{ on } \p\disk,
                    \end{cases}
                \]
                for $0 \leq k \leq n-1$, are solved by the complex components $w^+$ and $w^-$ of the idempotent representation of $w = p^+ w^+ + p^- w^-$. 
            \end{corr}

            \begin{corr}
                Let $n$ be a positive integer, $f \in L^1(\disk, \bc)$, and $\gamma_k \in C(\p\disk,\bc)$, for $0 \leq k \leq n-1$. The Dirichlet problem for 
                \begin{equation}\label{eq: eqcor4.7}
                    \begin{cases}
                        \dbar^n w = f, & \text{ in } \disk,\\
                        \dbar^k w = \gamma_k, & \text{ on } \p\disk,
                    \end{cases}
                \end{equation}
                for $0 \leq k \leq n-1$, is solvable if and only if, for $z \in \disk$ and $0 \leq k \leq n-1$, 
                \[
                    \sum_{\ell = k}^{n-1} \frac{z^*}{2\pi i} \int_{|\zeta|=1} (-1)^{\ell - k}\frac{(\gamma_k^+)^*(\zeta)}{1-z^*\zeta} \frac{[(\zeta - z)^*]^{\ell - k}}{(\ell - k)!}\,d\zeta +\frac{(-1)^{n-k}z^*}{\pi} \iint_{|\zeta|<1} \frac{(f^+)^*(\zeta)}{1-z^*\zeta} \frac{[(\zeta-z)^*]^{n-1-k}}{(n-1-k)!}\,d\xi\,d\eta = 0
                \]
                and
                \[
                    \sum_{\ell = k}^{n-1} \frac{z^*}{2\pi i} \int_{|\zeta|=1} (-1)^{\ell - k}\frac{\gamma_k^-(\zeta)}{1-z^*\zeta} \frac{[(\zeta - z)^*]^{\ell - k}}{(\ell - k)!}\,d\zeta +\frac{(-1)^{n-k}z^*}{\pi} \iint_{|\zeta|<1} \frac{f^-(\zeta)}{1-z^*\zeta} \frac{[(\zeta-z)^*]^{n-1-k}}{(n-1-k)!}\,d\xi\,d\eta = 0,
                \]
                where $\zeta = \xi + i\eta$. 
                The solution is 
                \begin{align*}
                    &w(z) \\
                    &= p^+\left( \sum_{k=0}^{n-1} \frac{(-1)^k}{2\pi i}\int_{|\zeta|=1} \frac{(\gamma_k^+)^*(\zeta)}{k!} \frac{[(\zeta - z)^*]^k}{\zeta - z}\,d\zeta + \frac{(-1)^n}{\pi} \iint_{|\zeta|<1} \frac{(f^+)^*(\zeta)}{(n-1)!} \frac{[(\zeta -z)^*]^{n-1}}{\zeta - z}\,d\xi\,d\eta\right)^* \\
                    &\quad\quad+ p^- \left( \sum_{k=0}^{n-1} \frac{(-1)^k}{2\pi i}\int_{|\zeta|=1} \frac{\gamma_k^-(\zeta)}{k!} \frac{[(\zeta - z)^*]^k}{\zeta - z}\,d\zeta + \frac{(-1)^n}{\pi} \iint_{|\zeta|<1} \frac{f^-(\zeta)}{(n-1)!} \frac{[(\zeta -z)^*]^{n-1}}{\zeta - z}\,d\xi\,d\eta\right). 
                \end{align*}
            \end{corr}

            \begin{proof}

                By Corollary \ref{corr: bcdirichletiffcomplexdirichletforanyn}, the Dirichlet problem 
                \[
                    \begin{cases}
                        \dbar^n w = f, & \text{ in } \disk,\\
                        \dbar^k w = \gamma_k, & \text{ on } \p\disk,
                    \end{cases}
                \]
                for $0 \leq k \leq n-1$, is solvable if and only if the Dirichlet problems
                \[
                    \begin{cases}
                        \frac{\p^n (w^+)^*}{\p (z^*)^n} = (f^+)^*, & \text{ in } \disk,\\
                        \frac{\p^k (w^+)^*}{\p (z^*)^k} = (\gamma_k^+)^*, & \text{ on } \p\disk,
                    \end{cases}
                \]
                for $0 \leq k \leq n-1$, and
                \[
                    \begin{cases}
                         \frac{\p^n w^-}{\p (z^*)^n} = f^-, & \text{ in } \disk,\\
                        \frac{\p^k w^-}{\p (z^*)^k} = \gamma_k^-, & \text{ on } \p\disk,
                    \end{cases}
                \]
                for $0 \leq k \leq n-1$, are solved by the complex components $w^+$ and $w^-$ of the idempotent representation of $w = p^+ w^+ + p^- w^-$. By Theorem \ref{thm: complexhigherorderdirichletdisk}, the Dirichlet problems
                \[
                    \begin{cases}
                        \frac{\p^n (w^+)^*}{\p (z^*)^n} = (f^+)^*, & \text{ in } \disk,\\
                        \frac{\p^k (w^+)^*}{\p (z^*)^k} = (\gamma_k^+)^*, & \text{ on } \p\disk,
                    \end{cases}
                \]
                for $0 \leq k \leq n-1$, and
                \[
                    \begin{cases}
                         \frac{\p^n w^-}{\p (z^*)^n} = f^-, & \text{ in } \disk,\\
                        \frac{\p^k w^-}{\p (z^*)^k} = \gamma_k^-, & \text{ on } \p\disk,
                    \end{cases}
                \]
                for $0 \leq k \leq n-1$, are solved by $(w^+)^*$ and $w^-$ if and only if, for $z \in \disk$ and $0 \leq k \leq n-1$, 
                \[
                    \sum_{\ell = k}^{n-1} \frac{z^*}{2\pi i} \int_{|\zeta|=1} (-1)^{\ell - k}\frac{(\gamma_k^+)^*(\zeta)}{1-z^*\zeta} \frac{[(\zeta - z)^*]^{\ell - k}}{(\ell - k)!}\,d\zeta +\frac{(-1)^{n-k}z^*}{\pi} \iint_{|\zeta|<1} \frac{(f^+)^*(\zeta)}{1-z^*\zeta} \frac{[(\zeta-z)^*]^{n-1-k}}{(n-1-k)!}\,d\xi\,d\eta = 0
                \]
                and
                \[
                    \sum_{\ell = k}^{n-1} \frac{z^*}{2\pi i} \int_{|\zeta|=1} (-1)^{\ell - k}\frac{\gamma_k^-(\zeta)}{1-z^*\zeta} \frac{[(\zeta - z)^*]^{\ell - k}}{(\ell - k)!}\,d\zeta +\frac{(-1)^{n-k}z^*}{\pi} \iint_{|\zeta|<1} \frac{f^-(\zeta)}{1-z^*\zeta} \frac{[(\zeta-z)^*]^{n-1-k}}{(n-1-k)!}\,d\xi\,d\eta = 0,
                \]
                where $\zeta = \xi + i\eta$. Also by Theorem \ref{thm: complexhigherorderdirichletdisk}, the solutions $(w^+)^*$ and $w^-$ are 
                \[
                    (w^+)^*(z) =   \sum_{k=0}^{n-1} \frac{(-1)^k}{2\pi i}\int_{|\zeta|=1} \frac{(\gamma_k^+)^*(\zeta)}{k!} \frac{[(\zeta - z)^*]^k}{\zeta - z}\,d\zeta + \frac{(-1)^n}{\pi} \iint_{|\zeta|<1} \frac{(f^+)^*(\zeta)}{(n-1)!} \frac{[(\zeta -z)^*]^{n-1}}{\zeta - z}\,d\xi\,d\eta\
                \]
                and 
                \[
                    w^-(z) = \sum_{k=0}^{n-1} \frac{(-1)^k}{2\pi i}\int_{|\zeta|=1} \frac{\gamma_k^-(\zeta)}{k!} \frac{[(\zeta - z)^*]^k}{\zeta - z}\,d\zeta + \frac{(-1)^n}{\pi} \iint_{|\zeta|<1} \frac{f^-(\zeta)}{(n-1)!} \frac{[(\zeta -z)^*]^{n-1}}{\zeta - z}\,d\xi\,d\eta.
                \]
                Since $w = p^+ w^- + p^- w^-$, it follows that the solution $w$ of the bicomplex problem \eqref{eq: eqcor4.7} is 
                \begin{align*}
                    &w(z) \\
                    &= p^+\left( \sum_{k=0}^{n-1} \frac{(-1)^k}{2\pi i}\int_{|\zeta|=1} \frac{(\gamma_k^+)^*(\zeta)}{k!} \frac{[(\zeta - z)^*]^k}{\zeta - z}\,d\zeta + \frac{(-1)^n}{\pi} \iint_{|\zeta|<1} \frac{(f^+)^*(\zeta)}{(n-1)!} \frac{[(\zeta -z)^*]^{n-1}}{\zeta - z}\,d\xi\,d\eta\right)^* \\
                    &\quad\quad+ p^- \left( \sum_{k=0}^{n-1} \frac{(-1)^k}{2\pi i}\int_{|\zeta|=1} \frac{\gamma_k^-(\zeta)}{k!} \frac{[(\zeta - z)^*]^k}{\zeta - z}\,d\zeta + \frac{(-1)^n}{\pi} \iint_{|\zeta|<1} \frac{f^-(\zeta)}{(n-1)!} \frac{[(\zeta -z)^*]^{n-1}}{\zeta - z}\,d\xi\,d\eta\right). 
                \end{align*}
                Therefore, the bicomplex Dirichlet problem 
                \[
                    \begin{cases}
                        \dbar^n w = f, & \text{ in } \disk,\\
                        \dbar^k w = \gamma_k, & \text{ on } \p\disk,
                    \end{cases}
                \]
                for $0 \leq k \leq n-1$, is solvable if and only, for $z \in \disk$ and $0 \leq k \leq n-1$, 
                \[
                    \sum_{\ell = k}^{n-1} \frac{z^*}{2\pi i} \int_{|\zeta|=1} (-1)^{\ell - k}\frac{(\gamma_k^+)^*(\zeta)}{1-z^*\zeta} \frac{[(\zeta - z)^*]^{\ell - k}}{(\ell - k)!}\,d\zeta +\frac{(-1)^{n-k}z^*}{\pi} \iint_{|\zeta|<1} \frac{(f^+)^*(\zeta)}{1-z^*\zeta} \frac{[(\zeta-z)^*]^{n-1-k}}{(n-1-k)!}\,d\xi\,d\eta = 0
                \]
                and
                \[
                    \sum_{\ell = k}^{n-1} \frac{z^*}{2\pi i} \int_{|\zeta|=1} (-1)^{\ell - k}\frac{\gamma_k^-(\zeta)}{1-z^*\zeta} \frac{[(\zeta - z)^*]^{\ell - k}}{(\ell - k)!}\,d\zeta +\frac{(-1)^{n-k}z^*}{\pi} \iint_{|\zeta|<1} \frac{f^-(\zeta)}{1-z^*\zeta} \frac{[(\zeta-z)^*]^{n-1-k}}{(n-1-k)!}\,d\xi\,d\eta = 0,
                \]
                where $\zeta = \xi + i\eta$, and the solution is 
                \begin{align*}
                    &w(z) \\
                    &= p^+\left( \sum_{k=0}^{n-1} \frac{(-1)^k}{2\pi i}\int_{|\zeta|=1} \frac{(\gamma_k^+)^*(\zeta)}{k!} \frac{[(\zeta - z)^*]^k}{\zeta - z}\,d\zeta + \frac{(-1)^n}{\pi} \iint_{|\zeta|<1} \frac{(f^+)^*(\zeta)}{(n-1)!} \frac{[(\zeta -z)^*]^{n-1}}{\zeta - z}\,d\xi\,d\eta\right)^* \\
                    &\quad\quad+ p^- \left( \sum_{k=0}^{n-1} \frac{(-1)^k}{2\pi i}\int_{|\zeta|=1} \frac{\gamma_k^-(\zeta)}{k!} \frac{[(\zeta - z)^*]^k}{\zeta - z}\,d\zeta + \frac{(-1)^n}{\pi} \iint_{|\zeta|<1} \frac{f^-(\zeta)}{(n-1)!} \frac{[(\zeta -z)^*]^{n-1}}{\zeta - z}\,d\xi\,d\eta\right). 
                \end{align*}
            \end{proof}

            \begin{theorem}\label{thm: bcdirichletsolvabletfae}
                 Let $n$ be a positive integer, $f \in L^1(\disk, \bc)$, and $\gamma_k \in C(\p\disk, \bc)$. The following three statements are equivalent for a function $w = p^+ w^+ + p^- w^-: \disk \to\bc$. 
                \begin{itemize}
                \item The Dirichlet problem
                \[
                    \begin{cases}
                        \dbar^n w = f, & \text{ in } \disk,\\
                        \dbar^k w = \gamma_k, & \text{ on } \p\disk,
                    \end{cases}
                \]
                for $0 \leq k \leq n-1$, is solvable, and the solution is $w$.
                \item The Dirichlet problems
                \[
                    \begin{cases}
                        \frac{\p^n (w^+)^*}{\p (z^*)^n} = (f^+)^*, & \text{ in } \disk,\\
                        \frac{\p^k (w^+)^*}{\p (z^*)^k} = (\gamma_k^+)^*, & \text{ on } \p\disk,
                    \end{cases}
                \]
                for $0 \leq k \leq n-1$, and
                \[
                    \begin{cases}
                         \frac{\p^n w^-}{\p (z^*)^n} = f^-, & \text{ in } \disk,\\
                        \frac{\p^k w^-}{\p (z^*)^k} = \gamma_k^-, & \text{ on } \p\disk,
                    \end{cases}
                \]
                for $0 \leq k \leq n-1$, are solvable, and the solutions are $(w^+)^*$ and $w^-$, respectively.
                \item For $z \in \disk$ and $0 \leq k \leq n-1$, 
                \[
                    \sum_{\ell = k}^{n-1} \frac{z^*}{2\pi i} \int_{|\zeta|=1} (-1)^{\ell - k}\frac{(\gamma_k^+)^*(\zeta)}{1-z^*\zeta} \frac{[(\zeta - z)^*]^{\ell - k}}{(\ell - k)!}\,d\zeta +\frac{(-1)^{n-k}z^*}{\pi} \iint_{|\zeta|<1} \frac{(f^+)^*(\zeta)}{1-z^*\zeta} \frac{[(\zeta-z)^*]^{n-1-k}}{(n-1-k)!}\,d\xi\,d\eta = 0
                \]
                and
                \[
                    \sum_{\ell = k}^{n-1} \frac{z^*}{2\pi i} \int_{|\zeta|=1} (-1)^{\ell - k}\frac{\gamma_k^-(\zeta)}{1-z^*\zeta} \frac{[(\zeta - z)^*]^{\ell - k}}{(\ell - k)!}\,d\zeta +\frac{(-1)^{n-k}z^*}{\pi} \iint_{|\zeta|<1} \frac{f^-(\zeta)}{1-z^*\zeta} \frac{[(\zeta-z)^*]^{n-1-k}}{(n-1-k)!}\,d\xi\,d\eta = 0,
                \]
                where $\zeta = \xi + i\eta$.
                \end{itemize}
                When it exists, the function $w$ is
                \begin{align*}
                    &w(z) \\
                    &= p^+\left( \sum_{k=0}^{n-1} \frac{(-1)^k}{2\pi i}\int_{|\zeta|=1} \frac{(\gamma_k^+)^*(\zeta)}{k!} \frac{[(\zeta - z)^*]^k}{\zeta - z}\,d\zeta + \frac{(-1)^n}{\pi} \iint_{|\zeta|<1} \frac{(f^+)^*(\zeta)}{(n-1)!} \frac{[(\zeta -z)^*]^{n-1}}{\zeta - z}\,d\xi\,d\eta\right)^* \\
                    &\quad\quad+ p^- \left( \sum_{k=0}^{n-1} \frac{(-1)^k}{2\pi i}\int_{|\zeta|=1} \frac{\gamma_k^-(\zeta)}{k!} \frac{[(\zeta - z)^*]^k}{\zeta - z}\,d\zeta + \frac{(-1)^n}{\pi} \iint_{|\zeta|<1} \frac{f^-(\zeta)}{(n-1)!} \frac{[(\zeta -z)^*]^{n-1}}{\zeta - z}\,d\xi\,d\eta\right). 
                \end{align*}
            \end{theorem}

        \subsubsection{Other Non-degenerate Conics}

            Next, we show the conclusion of Theorem  \ref{thm: dirichletbianalyticwellposed} also holds in the setting of the $\bc$-Bitsadze equation. That is, despite not being well-defined when the boundary condition is on the unit circle, the Dirichlet problem for the $\bc$-Bitsadze equation is well-posed when the boundary condition is a nondegenerate conic that is not a circumference. We prove this by appealing to the structure of solutions to the $\bc$-Bitsadze equation (which is similar to that of its complex counterpart), the idempotent representation for bicomplex numbers, and Theorem \ref{thm: dirichletbianalyticwellposed}. 

            \begin{theorem}\label{thm: generalizationOfBianalyticWellPosed}
             Let $Q(x,y) = 0$ be a non-degenerate conic that is not a circumference, $\Gamma = \{(x,y) \in \mathbb{R}^2 : Q(x,y) = 0\}$, and $P$ be a polynomial in $x$ and $y$. The Dirichlet problem for the $\bc$-Bitsadze equation
                \begin{equation}\label{eq: bicomplexdirichletthm}
                    \begin{cases}
                        \dbar^2 w  = 0, & \text{ in } \mathbb{R}^2\setminus \Gamma,\\
                        w = P, & \text{ on } \Gamma,
                    \end{cases}
                \end{equation}
                has a unique solution.
                \end{theorem}

                \begin{proof}

                By Lemma \ref{lem: bicomplexpolyanalyticidempotentrep}, every solution $w = p^+ w^+ + p^- w^-$ of $\dbar^2 w = 0$ satisfies 
                \[
                    \frac{\p^2 (w^+)^*}{\p (z^*)^2} = 0 = \frac{\p^2 w^-}{\p (z^*)^2}.
                \]
                By Theorem \ref{thm: dirichletbianalyticwellposed}, there exist unique functions $f, g$ that solve the Dirichlet problems for the Bitsadze equation
                \[
                    \begin{cases}
                        \frac{\p^2 f}{\p(z^*)^2}  = 0, & \text{ in } \mathbb{R}^2\setminus \Gamma,\\
                        f = (P^+)^*, & \text{ on } \Gamma,
                    \end{cases}
                \]
                and 
               \[
                    \begin{cases}
                        \frac{\p^2 g}{\p(z^*)^2}  = 0, & \text{ in } \mathbb{R}^2\setminus \Gamma,\\
                        g = P^-, & \text{ on } \Gamma,
                    \end{cases}
                \]
                respectively, where $P^+$ and $P^-$ are the components of the idempotent representation of the polynomial $P = p^+ P^+ + p^- P^-$. If we define $w:= p^+(f)^* + p^- g$, then $w$ solves $\dbar^2 w = 0$ and 
                \begin{align*}
                    w|_{\Gamma} &= p^+ (f)^*|_{\Gamma} + p^- g|_{\Gamma} \\
                    &= p^+ ((P^+)^*)^* + p^- P^-\\
                    &= p^+ P^+ + p^- P^- = P.
                \end{align*}
                By the uniqueness of $f$ and $g$ and the uniqueness of the idempotent representation of $w$, it follows that the constructed solution $w$ of \eqref{eq: bicomplexdirichletthm} is unique.

            \end{proof}

            \begin{comm}
                Note that the argument used in the proof of Theorem \ref{thm: generalizationOfBianalyticWellPosed} will immediately extend to bicomplex Dirichlet problems associated with solutions of $\dbar^n w =0$, for arbitrary positive integer $n$, when an analogous result to Theorem \ref{thm: dirichletbianalyticwellposed} for higher-order complex equations is available. 
            \end{comm}

    \subsection{The Dirichlet Problem for the Bicomplex Vekua-Bitsadze Equation}

        With the results for the Dirichlet problem for the Bitsadze equation extended to the Dirichlet problem for the $\bc$-Bitsadze equation, we work to show that the results from Section \ref{section: bcdirifchletbitsadzeequation} also hold for the $\bc$-Vekua-Bitsadze equation.

        \subsubsection{The Unit Circle}

        We begin with a lemma that extends Lemma \ref{lem: bicomplexpolyanalyticidempotentrep} to the bicomplex higher-order iterated Vekua equations. After the lemma, we extend Proposition \ref{prop: bicomplexbitsadzedirichlet} and its corollary to bicomplex higher-order Vekua equations.

            \begin{lem}\label{lem: bicomplexhoividempotentrep}
                Let $n$ be a positive integer and $A \in W^{n-1, q}(\disk,\bc)$, $q>2$. A function $w = p^+ w^+ + p^- w^-: \disk \to \bc$ is a solution to
                \[
                    (\dbar- A)^n w = 0
                \]
                if and only if $(w^+)^*$ and $w^-$ solve 
                \[
                   \left( \frac{\p }{\p z^*}- (A^+)^*\right)^n (w^+)^* = 0
                \]
                and 
                \[
                    \left( \frac{\p}{\p z^*} - A^-\right)^n w^- = 0,
                \]
                respectively. 
            \end{lem}

            \begin{proof}
                Suppose $w: \disk\to\bc$ solves
                \[
                    (\dbar- A)^n w = 0.
                \]
                By Theorem \ref{thm: 6point3BCHoiv}, 
                \[
                    w(z) = \sum_{k=0}^{n-1} \widehat{z^*}^k \varphi_k(z),
                \]
                where $\varphi_k$ solves
                \[
                    \dbar \varphi_k = A \varphi_k,
                \]
                for every $k$. By Theorem \ref{thm: 4.9BCHoiv}, $(\varphi_k^+)^*, \varphi_k^-$ solve
                \[
                 \frac{\p (\varphi_k^+)^*}{\p z^*} = (A^+)^* (\varphi_k^+)^*
                \]
                and 
                \[
                    \frac{\p \varphi_k^-}{\p z^*} = A^-\varphi_k^-,
                \]
                respectively, for every $k$. Since $\widehat{z^*} = p^+ z + p^- z^*$, it follows that 
                \begin{align*}
                     w(z) &= \sum_{k=0}^{n-1} \widehat{z^*}^k \varphi_k(z)\\
                     &= \sum_{k=0}^{n-1} (p^+ z + p^- z^*)^k (p^+ \varphi_k^+ + p^- \varphi_k^-) \\
                     &= \sum_{k=0}^{n-1} (p^+ z^k + p^- (z^*)^k) (p^+ \varphi_k^+ + p^- ]\varphi_k^-) \\
                     &= \sum_{k=0}^{n-1} \left(p^+ z^k \varphi_k^+ + p^- (z^*)^k \varphi_k^-\right) \\
                     &= p^+ \sum_{k=0}^{n-1} z^k \varphi_k^+ + p^- \sum_{k=0}^{n-1} (z^*)^k \varphi_k^-.
                \end{align*}
                Thus, 
                \[
                    w^+ = \sum_{k=0}^{n-1} z^k \varphi_k^+ ,
                \]
                and
                \[
                    w^- = \sum_{k=0}^{n-1} (z^*)^k \varphi_k^-.
                \]
                By Corollary \ref{cor: altbzerohoivrep}, $(w^+)^*, w^-$ are solutions of the higher-order iterated Vekua equations
                \[
                    \left( \frac{\p }{\p z^*} - (A^+)^*\right)^n (w^+)^* = 0
                \]
                and 
                \[
                \left( \frac{\p}{\p z^*} - A^-\right)^n w^-= 0,
                \]
                respectively.  

                Next, suppose $f, g : \disk \to \mathbb{C}$ solve the higher-order iterated Vekua equations
                \[
                    \left( \frac{\p }{\p z^*} - (A^+)^*\right)^n f = 0
                \]
                and 
                \[
                \left( \frac{\p}{\p z^*} - A^-\right)^n g= 0,
                \]
                respectively. Construct the function $w: \disk \to \bc$ by 
                \[
                    w(z) := p^+ (f(z))^* + p^- g(z).
                \]
                Observe that 
                \begin{align*}
                    (\dbar - A)^n w &= \left( p^+ \frac{\p}{\p z} + p^- \frac{\p}{\p z^*}- (p^+ A^+ + p^- A^-)\right)^n (p^+ (f)^* + p^- g)\\
                    &= \left( p^+ \left(\frac{\p}{\p z} - A^+\right)+ p^- \left(\frac{\p}{\p z^*} - A^-\right)\right)^n (p^+ (f)^* + p^- g)\\
                    &= \left( p^+  \left(\frac{\p}{\p z} - A^+\right)^n + p^- \left(\frac{\p}{\p z^*} - A^-\right)^n \right)(p^+ (f)^* + p^- g)\\
                    &= p^+ \left(\frac{\p}{\p z} - A^+\right)^n (f)^* + p^- \left(\frac{\p}{\p z^*} - A^-\right)^n g \\
                    &= p^+ \left[\left(\frac{\p}{\p z^*} - (A^+)^*\right)^n f \right]^* + p^- \left(\frac{\p}{\p z^*} - A^-\right)^n g \\
                    &= p^+ (0) + p^- ( 0) = 0. 
                \end{align*}
                Therefore, $w$ solves 
                \[
                    (\dbar-A)^n w = 0.
                \]
            \end{proof}

            \begin{theorem}\label{thm: bicomplexvekuabitsadzedirichletnotwelldefined}
            Let $A \in W^{1,q}(\disk,\bc)$, $q>2$. The Dirichlet problem 
            \begin{equation}\label{eq: bicomplexvekbitsadzedirchilet}
            \begin{cases}
                \left( \dbar - A \right)^2 w = 0, & \text{ in } \disk,\\
                w = 0, & \text{ on } \p\disk,
            \end{cases}
            \end{equation}
            has infinitely many solutions. 
            \end{theorem}

            \begin{proof}

                 By Lemma \ref{lem: bicomplexhoividempotentrep}, $w = p^+ w^+ + p^-w^-$ is a solution of $(\dbar - A)^2 w = 0$ if and only if $(w^+)^*$ and $w^-$ solve
                 \[
                        \left( \frac{\p }{\p z^*} - (A^+)^*\right)^2 (w^+)^* = 0
                 \]
                 and 
                 \[ 
                        \left( \frac{\p}{\p z^*} - A^-\right)^2 w^- = 0,
                 \]
                 respectively. Since $A \in W^{1,q}(\disk, \bc)$, $q>2$, it follows that $(A^+)^*, A^- \in W^{1, q}(\disk)$ by Proposition \ref{cor: corollaryof2.24BCAtomic}. So, by Proposition \ref{lem: generalbitsadzelindep}, there exist infinitely many solutions $f$ and $g$ to the Dirichlet problems 
                 \begin{equation}\label{eq: hoivinthm1}
                    \begin{cases}
                        \left( \frac{\p }{\p z^*} - (A^+)^*\right)^2 f = 0, & \text{ in } \disk, \\
                        f = 0, & \text{ on }\p\disk,
                    \end{cases}
                 \end{equation}
                 and
                 \begin{equation}\label{eq: hoivinthm2}
                    \begin{cases}
                        \left( \frac{\p}{\p z^*} - A^-\right)^2 g = 0,& \text{ in }\disk,\\
                        g = 0, & \text{ on } \p\disk,
                    \end{cases}
                 \end{equation}
                 respectively. Define a function $w:= p^+ (f^+)^* + p^- g$, where $f$ and $g$ are any two functions that solve \eqref{eq: hoivinthm1} and \eqref{eq: hoivinthm2}, respectively. Observe that 
                 \begin{align*}
                    (\dbar - A)^2 w 
                    &=\left( p^+ \frac{\p}{\p z} + p^- \frac{\p}{\p z^*}- (p^+ A^+ + p^- A^-)\right )^2 (p^+ (f)^* + p^- g) \\
                    &= \left( p^+ \left(\frac{\p}{\p z}- A^+ \right) + p^- \left(\frac{\p}{\p z^*}-  A^-\right)\right )^2 (p^+ (f)^* + p^- g) \\
                    &=  \left(p^+ \left(\frac{\p}{\p z}- A^+ \right)^2 + p^- \left(\frac{\p}{\p z^*}-  A^-\right)^2\right) (p^+ (f)^* + p^- g) \\
                    &= p^+ \left(\frac{\p}{\p z}- A^+ \right)^2 (f)^* + p^-  \left(\frac{\p}{\p z^*}-  A^-\right)^2 g \\
                    &= p^+\left(  \left(\frac{\p}{\p z^*}- (A^+)^* \right)^2 f \right)^*+ p^-  \left(\frac{\p}{\p z^*}-  A^-\right)^2 g =0
                 \end{align*}
                 and 
                 \begin{align*}
                    w|_{\p \disk} &= p^+ (f)^*|_{\p\disk} + p^- g|_{\p\disk} \\
                    &= p^+(0) + p^- (0) = 0.
                 \end{align*}
                 Thus, $w = p^+ (f)^* + p^- g$ solves \eqref{eq: bicomplexvekbitsadzedirchilet}. Since there are infinitely many pairs of functions $f$ and $g$ that solve \eqref{eq: hoivinthm1} and \eqref{eq: hoivinthm2}, respectively, and each pair produces a function $w = p^+ (f)^* + p^- g$ that solves \eqref{eq: bicomplexvekbitsadzedirchilet}, it follows that there are infinitely many solutions of \eqref{eq: bicomplexvekbitsadzedirchilet}. 
            
            \end{proof}

             \begin{corr}
            Let $n$ be a positive integer and $A \in W^{n-1,q}(\disk,\bc)$, $q>2$. The Dirichlet problem 
            \[
            \begin{cases}
                \left( \dbar - A \right)^n w = 0, & \text{ in } \disk,\\
                w = 0, & \text{ on } \p\disk,
            \end{cases}
            \]
            has infinitely many solutions. 
            \end{corr}

            \begin{proof}
                Use the same argument as the proof of Theorem \ref{thm: bicomplexvekuabitsadzedirichletnotwelldefined} except appeal to Proposition \ref{prop: generalbitsadzelindepextension} in the place of Proposition \ref{lem: generalbitsadzelindep}. 
            \end{proof}

            Now, we extend Proposition \ref{prop: bcdirichletiffcomplexdirichletfornequalstwo}, its two corollaries, and Theorem \ref{thm: bcdirichletsolvabletfae}, which summarizes the preceding statements, to $\bc$-HOIV equations. These results provide conditions for solvability of the Dirichlet problem for the $\bc$-Vekua-Bitsadze equation and its higher-order extensions on the disk. 

            \begin{prop}
                Let $A \in W^{1,q}(\disk,\bc)$, $q>2$, and $\gamma_0, \gamma_1 \in C(\p\disk,\bc)$. The Dirichlet problem for the $\bc$-Vekua-Bitsadze equation
                \[
                    \begin{cases}
                        (\dbar - A)^2 w = 0, & \text{ in } \disk, \\
                        w = \gamma_0, & \text{ on } \p\disk,\\
                        (\dbar - A)w = \gamma_1, & \text{ on } \p\disk,
                    \end{cases}
                \]
                is solvable if and only if the Dirichlet problems for the Vekua-Bitsadze equation
                \[
                    \begin{cases}
                        \left( \frac{\p}{\p z^*} - (A^+)^*\right)^2 (w^+)^* = 0, & \text{ in } \disk,\\
                        (w^+)^* = (\gamma_0^+)^*, &\text{ on } \p\disk,\\
                        \left( \frac{\p}{\p z^*} - (A^+)^*\right) (w^+)^* = (\gamma_1^+)^*, & \text{ on } \p\disk,
                    \end{cases}
                \]
                and 
                \[
                    \begin{cases}
                        \left( \frac{\p}{\p z^*} - A^-\right)^2 w^- = 0, & \text{ in } \disk,\\
                        w^- = \gamma_0^-, &\text{ on } \p\disk,\\
                        \left( \frac{\p}{\p z^*} - A^-\right) w^- = \gamma_1^-, & \text{ on } \p\disk,
                    \end{cases}
                \]
                are solved by the complex components $w^+$ and $w^-$ of the idempotent representation of $w = p^+ w^+ + p^- w^-$. 
            \end{prop}

            \begin{proof}
                Recall, by Proposition \ref{prop: bcctsiffidemc}, $\gamma_0,\gamma_1 \in C(\p\disk,\bc)$ if and only if $(\gamma_0)^*, \gamma_0^-, (\gamma_1^+)^*, \gamma_1^- \in C(\p\disk)$. By Lemma \ref{lem: bicomplexhoividempotentrep}, 
                \[
                    (\dbar - A)^2 w = 0
                \]
                if and only if 
                \[
                     \left( \frac{\p}{\p z^*} - (A^+)^*\right)^2 (w^+)^* = 0
                \]
                and 
                \[
                     \left( \frac{\p}{\p z^*} - A^-\right)^2 w^- = 0.
                \]
                Since 
                \begin{align*}
                    w|_{\p\disk} &= p^+ w^+|_{\p\disk} + p^- w^-|_{\p\disk} 
                \end{align*}
                and
                \begin{align*}
                    \gamma_0 &= p^+ \gamma_0^+ + p^- \gamma_0^-,
                \end{align*}
                it follows that $w= \gamma_0$ on $\p\disk$ if and only if $(w^+)^*|_{\p\disk} = (\gamma_0^+)^*$ and $w^-|_{\p\disk} = \gamma_0^-$. Similarly, since 
                \[
                    (\dbar - A) w = p^+ \left(\frac{\p w^+}{\p z}-A^+\right)w^+ + p^- \left(\frac{\p w^-}{\p z^*}-A^-\right) w^-
                \]
                and 
                \begin{align*}
                    \gamma_1 &= p^+ \gamma_1^+ + p^- \gamma_1^-,
                \end{align*}
                it follows that $(\dbar-A) w = \gamma_1$ on $\p\disk$ if and only if $\left(\frac{\p }{\p z^*} -(A^+)^*\right) (w^+)^*|_{\p\disk}= (\gamma_1^+)^*$ and $\left(\frac{\p }{\p z^*}-A^-\right)w^-|_{\p\disk} = \gamma_1^-$. 

            \end{proof}

            \begin{corr}\label{corr: hoivdirichletforalln}
                Let $n$ be a positive integer, $A \in W^{n-1,q}(\disk,\bc)$, and $\gamma_k \in C(\p\disk,\bc)$, for $0 \leq k \leq n-1$. The Dirichlet problem for the bicomplex higher-order iterated Vekua equation
                \[
                    \begin{cases}
                        (\dbar-A)^n w = 0, & \text{ in } \disk,\\
                        (\dbar - A)^k w = \gamma_k, & \text{ on } \p\disk,
                    \end{cases}
                \]
                for $0 \leq k \leq n-1$, is solvable if and only if the Dirichlet problems for the higher-order iterated Vekua equations
                \[
                    \begin{cases}
                        \left( \frac{\p}{\p z^*} - (A^+)^*\right)^n (w^+)^* = 0, & \text{ in } \disk,\\
                        \left( \frac{\p}{\p z^*} - (A^+)^*\right)^k (w^+)^* = (\gamma_k^+)^*, & \text{ on } \p\disk,
                    \end{cases}
                \]
                for $0 \leq k \leq n-1$, and
                \[
                    \begin{cases}
                        \left( \frac{\p}{\p z^*} - A^-\right)^n w^- = 0, & \text{ in } \disk,\\
                        \left( \frac{\p}{\p z^*} - A^-\right)^k w^- = \gamma_k^-, & \text{ on } \p\disk,
                    \end{cases}
                \]
                for $0 \leq k \leq n-1$, are solved by the complex components $w^+$ and $w^-$ of the idempotent representation of $w = p^+ w^+ + p^- w^-$.          
            \end{corr}

            \begin{corr}
                Let $n$ be a positive integer, $A \in W^{n-1,q}(\disk,\bc)$, and $\gamma_k \in C(\p\disk,\bc)$, for $0 \leq k \leq n-1$. The Dirichlet problem for the bicomplex higher-order iterated Vekua equation
                \[
                    \begin{cases}
                        (\dbar-A)^n w = 0, & \text{ in } \disk,\\
                        (\dbar - A)^k w = \gamma_k, & \text{ on } \p\disk,
                    \end{cases}
                \]
                for $0 \leq k \leq n-1$, is solvable if and only, for $|z| < 1$  and $0 \leq k \leq n-1$,
                \[
                    \sum_{\lambda = k}^{n-1}\frac{z^*}{2\pi i} \int_{|\zeta| = 1} (-1)^{\lambda - k} \frac{(\gamma_\lambda^+)^*(\zeta)e^{-T_\disk[(A^+)^*](\zeta)}}{1- z^*\zeta} \frac{[(\zeta - z)^*]^{\lambda- k}}{(\lambda - k)!}\,d\zeta   = 0
                \]
                and
                \[
                    \sum_{\lambda = k}^{n-1}\frac{z^*}{2\pi i} \int_{|\zeta| = 1} (-1)^{\lambda - k} \frac{\gamma_\lambda^-(\zeta) e^{-T_\disk[A^-](\zeta)}}{1- z^*\zeta} \frac{[(\zeta - z)^*]^{\lambda- k}}{(\lambda - k)!}\,d\zeta   = 0.
                \]

                The solution is 
                \[
                    w = p^+ (e^{T_\disk[(A^+)^*]}  \varphi_+)^* + p^- e^{T_\disk[A^-]} \varphi_-, 
                \]
                where
                \[
                     \varphi_+(z) =  \sum_{k=0}^{n-1} \frac{(-1)^k}{2\pi i} \int_{|\zeta| = 1} \frac{(\gamma_k^+)^*(\zeta)e^{-T_\disk[(A^+)^*](\zeta)}}{k!} \frac{[(\zeta - z)^*]^k}{\zeta - z}\,d\zeta 
                \]
                and
                \[
                     \varphi_-(z) =  \sum_{k=0}^{n-1} \frac{(-1)^k}{2\pi i} \int_{|\zeta| = 1} \frac{\gamma_k^-(\zeta)e^{-T_\disk[A^-](\zeta)}}{k!} \frac{[(\zeta - z)^*]^k}{\zeta - z}\,d\zeta .
                \]
            \end{corr}

                \begin{proof}

                    By Corollary \ref{corr: hoivdirichletforalln}, the Dirichlet problem for the $\bc$-HOIV equation
                \[
                    \begin{cases}
                        (\dbar-A)^n w = 0, & \text{ in } \disk,\\
                        (\dbar - A)^k w = \gamma_k, & \text{ on } \p\disk,
                    \end{cases}
                \]
                for $0 \leq k \leq n-1$, is solvable if and only if the Dirichlet problems for HOIV equations
                \[
                    \begin{cases}
                        \left( \frac{\p}{\p z^*} - (A^+)^*\right)^n (w^+)^* = 0, & \text{ in } \disk,\\
                        \left( \frac{\p}{\p z^*} - (A^+)^*\right)^k (w^+)^* = (\gamma_k^+)^*, & \text{ on } \p\disk,
                    \end{cases}
                \]
                for $0 \leq k \leq n-1$, and
                \[
                    \begin{cases}
                        \left( \frac{\p}{\p z^*} - A^-\right)^n w^- = 0, & \text{ in } \disk,\\
                        \left( \frac{\p}{\p z^*} - A^-\right)^k w^- = \gamma_k^-, & \text{ on } \p\disk,
                    \end{cases}
                \]
                for $0 \leq k \leq n-1$, are solved by the complex components $w^+$ and $w^-$ of the idempotent representation of $w = p^+ w^+ + p^- w^-$. By Theorem \ref{thm: thmfourpointfivewbd2}, the Dirichlet problems
                \[
                    \begin{cases}
                        \left( \frac{\p}{\p z^*} - (A^+)^*\right)^n (w^+)^* = 0, & \text{ in } \disk,\\
                        \left( \frac{\p}{\p z^*} - (A^+)^*\right)^k (w^+)^* = (\gamma_k^+)^*, & \text{ on } \p\disk,
                    \end{cases}
                \]
                for $0 \leq k \leq n-1$, and
                \[
                    \begin{cases}
                        \left( \frac{\p}{\p z^*} - A^-\right)^n w^- = 0, & \text{ in } \disk,\\
                        \left( \frac{\p}{\p z^*} - A^-\right)^k w^- = \gamma_k^-, & \text{ on } \p\disk,
                    \end{cases}
                \]
                for $0 \leq k \leq n-1$, are solvable if and only if, for $|z| < 1$  and $0 \leq k \leq n-1$,
            \[
                \sum_{\lambda = k}^{n-1}\frac{z^*}{2\pi i} \int_{|\zeta| = 1} (-1)^{\lambda - k} \frac{(\gamma^+_\lambda)^*(\zeta)e^{-T_\disk[(A^+)^*](\zeta)}}{1- z^*\zeta} \frac{[(\zeta - z)^*]^{\lambda- k}}{(\lambda - k)!}\,d\zeta   = 0.
            \]
            and
            \[
                \sum_{\lambda = k}^{n-1}\frac{z^*}{2\pi i} \int_{|\zeta| = 1} (-1)^{\lambda - k} \frac{\gamma^-_\lambda(\zeta)e^{-T_\disk[A^-](\zeta)}}{1- z^*\zeta} \frac{[(\zeta - z)^*]^{\lambda- k}}{(\lambda - k)!}\,d\zeta   = 0.
            \]
            Also, by Theorem \ref{thm: thmfourpointfivewbd2}, the solutions $(w^+)^*$ and $w^-$ are 
            \[
                (w^+)^* = e^{T_\disk[(A^+)^*]}  \varphi_+, 
            \]
            where
            \[
                \varphi_+(z) =  \sum_{k=0}^{n-1} \frac{(-1)^k}{2\pi i} \int_{|\zeta| = 1} \frac{(\gamma^+_\lambda)^*(\zeta)e^{-T_\disk[(A^+)^*](\zeta)}}{k!} \frac{[(\zeta - z)^*]^k}{\zeta - z}\,d\zeta 
            \]
            and 
            \[
                w^- = e^{T_\disk[A^-]}  \varphi_-, 
            \]
            where
            \[
                \varphi_-(z) =  \sum_{k=0}^{n-1} \frac{(-1)^k}{2\pi i} \int_{|\zeta| = 1} \frac{\gamma^-_\lambda(\zeta)e^{-T_\disk[A^-](\zeta)}}{k!} \frac{[(\zeta - z)^*]^k}{\zeta - z}\,d\zeta .
            \]
            Therefore, 
            \[
                    \begin{cases}
                        (\dbar-A)^n w = 0, & \text{ in } \disk,\\
                        (\dbar - A)^k w = \gamma_k, & \text{ on } \p\disk,
                    \end{cases}
                \]
                for $0 \leq k \leq n-1$, is solvable if and only, for $|z| < 1$  and $0 \leq k \leq n-1$,
                \[
                    \sum_{\lambda = k}^{n-1}\frac{z^*}{2\pi i} \int_{|\zeta| = 1} (-1)^{\lambda - k} \frac{(\gamma_\lambda^+)^*(\zeta)e^{-T_\disk[(A^+)^*](\zeta)}}{1- z^*\zeta} \frac{[(\zeta - z)^*]^{\lambda- k}}{(\lambda - k)!}\,d\zeta   = 0
                \]
                and
                \[
                    \sum_{\lambda = k}^{n-1}\frac{z^*}{2\pi i} \int_{|\zeta| = 1} (-1)^{\lambda - k} \frac{\gamma_\lambda^-(\zeta) e^{-T_\disk[A^-](\zeta)}}{1- z^*\zeta} \frac{[(\zeta - z)^*]^{\lambda- k}}{(\lambda - k)!}\,d\zeta   = 0,
                \]
                and the solution is 
                \[
                    w = p^+ (e^{T_\disk[(A^+)^*]}  \varphi_+)^* + p^- e^{T_\disk[A^-]} \varphi_-, 
                \]
                where
                \[
                     \varphi_+(z) =  \sum_{k=0}^{n-1} \frac{(-1)^k}{2\pi i} \int_{|\zeta| = 1} \frac{(\gamma_k^+)^*(\zeta)e^{-T_\disk[(A^+)^*](\zeta)}}{k!} \frac{[(\zeta - z)^*]^k}{\zeta - z}\,d\zeta 
                \]
                and
                \[
                     \varphi_-(z) =  \sum_{k=0}^{n-1} \frac{(-1)^k}{2\pi i} \int_{|\zeta| = 1} \frac{\gamma_k^-(\zeta)e^{-T_\disk[A^-](\zeta)}}{k!} \frac{[(\zeta - z)^*]^k}{\zeta - z}\,d\zeta .
                \]
                \end{proof}

            \begin{theorem}
                 Let $n$ be a positive integer, $A \in W^{n-1,q}(\disk,\bc)$, and $\gamma_k \in C(\p\disk,\bc)$, for $0 \leq k \leq n-1$. 
                 The following three statements are equivalent for a function $w = p^+ w^+ + p^- w^-:\disk\to\bc$:
                    \begin{enumerate}
                        \item The Dirichlet problem for the bicomplex higher-order Vekua equation
                        \[
                            \begin{cases}
                            (\dbar - A)^n w = 0, & \text{ in } \disk, \\
                            (\dbar - A)^k w = \gamma_k, & \text{ on } \p\disk,
                            \end{cases}
                        \]
                        for $0 \leq k \leq n-1$, is solvable, and the solution is $w$.
                        
                        \item The Dirichlet problems for the higher-order iterated Vekua equation
                        \[
                            \begin{cases}
                                \left( \frac{\p}{\p z^*} - (A^+)^*\right)^n (w^+)^* = 0, & \text{ in } \disk,\\
                                \left( \frac{\p}{\p z^*} - (A^+)^*\right)^k (w^+)^* = (\gamma_k^+)^*, & \text{ on } \p\disk,
                            \end{cases}
                        \]
                        for $0 \leq k \leq n-1$, and
                        \[
                            \begin{cases}
                                \left( \frac{\p}{\p z^*} - A^-\right)^n w^- = 0, & \text{ in } \disk,\\
                                \left( \frac{\p}{\p z^*} - A^-\right)^k w^- = \gamma_k^-, & \text{ on } \p\disk,
                            \end{cases}
                        \]
                        for $0 \leq k \leq n-1$, are solvable, and the solutions are $(w^+)^*$ and $w^-$, respectively.
                        
                        \item For $|z| < 1$  and $0 \leq k \leq n-1$,
                        \[
                            \sum_{\lambda = k}^{n-1}\frac{z^*}{2\pi i} \int_{|\zeta| = 1} (-1)^{\lambda - k} \frac{(\gamma_\lambda^+)^*(\zeta)e^{-T_\disk[(A^+)^*](\zeta)}}{1- z^*\zeta} \frac{[(\zeta - z)^*]^{\lambda- k}}{(\lambda - k)!}\,d\zeta   = 0
                        \]
                        and
                     \[
                            \sum_{\lambda = k}^{n-1}\frac{z^*}{2\pi i} \int_{|\zeta| = 1} (-1)^{\lambda - k} \frac{\gamma_\lambda^-(\zeta) e^{-T_\disk[A^-](\zeta)}}{1- z^* \zeta} \frac{[(\zeta - z)^*]^{\lambda- k}}{(\lambda - k)!}\,d\zeta   = 0.
                    \]
                    \end{enumerate}
                    When it exists, the function $w$ is
                    \[
                        w = p^+ (e^{T_\disk[(A^+)^*]}  \varphi_+)^* + p^- e^{T_\disk[A^-]} \varphi_-, 
                    \]
                    where
                    \[
                        \varphi_+(z) =  \sum_{k=0}^{n-1} \frac{(-1)^k}{2\pi i} \int_{|\zeta| = 1} \frac{(\gamma_k^+)^*(\zeta)e^{-T_\disk[(A^+)^*](\zeta)}}{k!} \frac{[(\zeta - z)^*]^k}{\zeta - z}\,d\zeta 
                    \]
                    and
                    \[
                        \varphi_-(z) =  \sum_{k=0}^{n-1} \frac{(-1)^k}{2\pi i} \int_{|\zeta| = 1} \frac{\gamma_k^-(\zeta)e^{-T_\disk[A^-](\zeta)}}{k!} \frac{[(\zeta - z)^*]^k}{\zeta - z}\,d\zeta .
                    \]
                \end{theorem}

        \subsubsection{Other Non-degenerate Conics}\label{last one}

            Finally, we show the conclusion of Theorem \ref{thm: mainthm} extends to the setting of the $\bc$-Vekua-Bitsadze equation with two proofs. Note, the boundary condition is again a product of an exponential function and a polynomial. Unsurprisingly, we replace the classic $T_D(\cdot)$ operator, a right-inverse for $\frac{\p}{\p z^*}$, with its bicomplex counterpart $T_D^\bc(\cdot)$, which is a right-inverse for $\dbar$, in the exponential function factor.

        \begin{theorem}\label{thm: bcmainthm}
        Let $Q(x,y) = 0$ be a non-degenerate conic that is not a circumference, $\Gamma = \{(x,y) \in \mathbb{R}^2 : Q(x,y) = 0\}$, $D$ be the region enclosed by $\Gamma$, $P$ be a polynomial in $x$ and $y$, and $A \in W^{1,q}(D,\bc)$, $q>2$. The Dirichlet problem for the $\bc$-Vekua-Bitsadze equation
        \begin{equation}\label{eq: vekbitdirwellposed}
            \begin{cases}
                \left( \dbar - A \right)^2 w = 0, & \text{ in } D,\\
                w = e^{T_D^\bc[A]}P, & \text{ on } \Gamma,
            \end{cases}
        \end{equation}
        has a unique solution. 
    \end{theorem}

    \begin{proof}

        By Lemma \ref{lem: bicomplexhoividempotentrep}, $w = p^+ w^+ + p^-w^-$ is a solution of $(\dbar^2 - A)^2 w = 0$ if and only if $(w^+)^*$ and $w^-$ solve
                 \[
                        \left( \frac{\p }{\p z^*} - (A^+)^*\right)^2 (w^+)^* = 0
                 \]
                 and 
                 \[ 
                        \left( \frac{\p}{\p z^*} - A^-\right)^2 w^- = 0,
                 \]
                 respectively. Since $A \in W^{1,q}(D, \bc)$, $q>2$, it follows that $(A^+)^*, A^- \in W^{1, q}(D)$ by Proposition \ref{cor: corollaryof2.24BCAtomic}. So, by Theorem \ref{thm: mainthm}, there exists a unique solution $f$ to the Dirichlet problem
                 \begin{equation}\label{eq: fdirproblem}
                    \begin{cases}
                        \left(\frac{\p}{\p z^*} - (A^+)^* \right)^2 f = 0,& \text{ in } D,\\
                        f = e^{T_D[(A^+)^*]} (P^+)^*, & \text{ on } \Gamma,
                    \end{cases}
                 \end{equation}
                and a unique solution $g$ to the Dirichlet problem 
                \begin{equation}\label{eq: gdirproblem}
                    \begin{cases}
                        \left( \frac{\p}{\p z^*} - A^-\right)^2 g = 0,& \text{ in } D,\\
                        g = e^{T_D[A^-]} P^-, & \text{ on } \Gamma.
                    \end{cases}
                 \end{equation}
                 Define the function $w$ by $w:= p^+ (f)^* + p^- g$. By construction, 
                 \[
                    (\dbar - A)^2 w = 0
                 \]
                 and 
                \begin{align*}
                    w|_{\Gamma} &= 
                    p^+ (f)^*|_{\Gamma} + p^- g|_{\Gamma}\\
                    &= p^+ (e^{T_D[(A^+)^*]} (P^+)^*)^* + p^- e^{T_D[A^-]}P^-\\
                    &= p^+ e^{(T_D[(A^+)^*])^*} P^+ + p^- e^{T_D[A^-]} P^-\\
                    &= (p^+ e^{(T_D[(A^+)^*])^*}+p^- e^{T_D[A^-]})(p^+P^+ + p^- P^-)\\
                    &= e^{p^+ (T_D[(A^+)^*])^* + p^- T_D[A^-]}P\\
                    &= e^{T_D^\bc[A]} P .
                \end{align*}
                Hence, $w$ solves \eqref{eq: vekbitdirwellposed}. Uniqueness of the solution follows from the uniqueness of the solutions of \eqref{eq: fdirproblem} and \eqref{eq: gdirproblem} and of the idempotent representation of $w$. 
 
            \end{proof}

            \begin{proof}[Alternative Proof]
                All solutions of $(\dbar - A)^2 w = 0$ have the form 
                \[
                    w = e^{T_D^\bc[A]} (\varphi_0 + \widehat{z^*}\varphi_1),
                \]
                where $\dbar \varphi_i = A\varphi_i$, $i \in \{1,2\}$, by Theorem \ref{thm: bchoivsimprinreptheorem}. By Theorem \ref{thm: generalizationOfBianalyticWellPosed}, there exist unique $\varphi_0, \varphi_1$ such that $b = \varphi_0 + \widehat{z^*} \varphi_1$ solves
                \[
                    \begin{cases}
                        \dbar^2 b = 0 & \text{ in } D\\
                        b = P & \text{ on } \Gamma
                    \end{cases}.
                \]
                Therefore, $w = e^{T_D^\bc[A]} b$ solves $(\dbar - A)^2 w = 0$ and 
                \begin{align*}
                    w|_{\Gamma} &= e^{T_D^\bc[A]} b|_{\Gamma}
                    = e^{T_D^\bc[A]} P. 
                \end{align*}
                Uniqueness of the solution follows from the uniqueness of $b$. 
            \end{proof}

\begin{comm}
    Similarly to the complex case, see Remark \ref{remark: hoivdbvp}, the argument of the first proof of Theorem \ref{thm: bcmainthm} directly extends to Dirichlet problems for $\bc$-HOIV equations (order greater than two) if a higher-order generalization of Theorem \ref{thm: mainthm} is available, and the argument of the second proof directly extends to Dirichlet problems for $\bc$-HOIV equations if a generalization of Theorem \ref{thm: generalizationOfBianalyticWellPosed} is available for $\bc$-polyanalytic functions of order greater than two. 
\end{comm}

\section*{Acknowledgements}
The author thanks the anonymous referees for their careful reading of this manuscript. Their comments and suggestions improved the quality of the work presented.

The author also thanks Briceyda Delgado for sparking their interest in the Dirichlet problem for the Bitsadze equation while working on \cite{WBD2}.

\section*{Declarations}
\subsection*{Conflict of Interest}
The authors declare that they have no conflicts of interest.

\subsection*{Funding Statement}
The authors declare that there are no funders to report for this submission.

\printbibliography

@book {Vek,
    AUTHOR = {Vekua, I. N.},
     TITLE = {Generalized analytic functions},
 PUBLISHER = {Pergamon Press, London-Paris-Frankfurt; Addison-Wesley
              Publishing Co., Inc., Reading, Mass.},
      YEAR = {1962},
     PAGES = {xxix+668},
   MRCLASS = {30.81},
  MRNUMBER = {0150320},
MRREVIEWER = {A. E. Heins},
}

@book {Balk,
    AUTHOR = {Balk, M.},
     TITLE = {Polyanalytic functions},
    SERIES = {Mathematical Research},
    VOLUME = {63},
 PUBLISHER = {Akademie-Verlag, Berlin},
      YEAR = {1991},
     PAGES = {197},
      ISBN = {3-05-501292-5},
   MRCLASS = {30G20 (31A30)},
  MRNUMBER = {1184141},
MRREVIEWER = {Heinrich Begehr},
}

@article {metahardy,
    AUTHOR = {Ku, M. and He, F. and Wang, Y.},
     TITLE = {Riemann-{H}ilbert problems for {H}ardy space of meta-analytic
              functions on the unit disc},
   JOURNAL = {Complex Anal. Oper. Theory},
  FJOURNAL = {Complex Analysis and Operator Theory},
    VOLUME = {12},
      YEAR = {2018},
    NUMBER = {2},
     PAGES = {457--474},
      ISSN = {1661-8254},
   MRCLASS = {30E25 (30H10)},
  MRNUMBER = {3756167},
MRREVIEWER = {Kuzman Adzievski},
       DOI = {10.1007/s11785-017-0705-1},
       URL = {https://doi.org/10.1007/s11785-017-0705-1},
}

@book {BegBook,
    AUTHOR = {Begehr, H.},
     TITLE = {Complex analytic methods for partial differential equations},
      NOTE = {An introductory text},
 PUBLISHER = {World Scientific Publishing Co., Inc., River Edge, NJ},
      YEAR = {1994},
     PAGES = {x+273},
      ISBN = {981-02-1550-9},
   MRCLASS = {35C15 (30E20 35A20)},
  MRNUMBER = {1314196},
MRREVIEWER = {Guo Chun Wen},
       DOI = {10.1142/2162},
       URL = {https://doi.org/10.1142/2162},
}

@book {ellipquasi,
    AUTHOR = {Astala, K. and Iwaniec, T. and Martin, G.},
     TITLE = {Elliptic partial differential equations and quasiconformal
              mappings in the plane},
    SERIES = {Princeton Mathematical Series},
    VOLUME = {48},
 PUBLISHER = {Princeton University Press, Princeton, NJ},
      YEAR = {2009},
     PAGES = {xviii+677},
      ISBN = {978-0-691-13777-3},
   MRCLASS = {30C62 (30G20 35J46 35J60 35J92)},
  MRNUMBER = {2472875},
MRREVIEWER = {Olli Martio},
}

@article {BegBVPCA2,
    AUTHOR = {Begehr, H.},
     TITLE = {Boundary value problems in complex analysis. {II}},
   JOURNAL = {Bol. Asoc. Mat. Venez.},
  FJOURNAL = {Bolet\'in de la Asociaci\'on Matem\'atica Venezolana},
    VOLUME = {12},
      YEAR = {2005},
    NUMBER = {2},
     PAGES = {217--250},
      ISSN = {1315-4125},
   MRCLASS = {30G20 (30E20 30E25 35J25)},
  MRNUMBER = {2229766},
MRREVIEWER = {Michael\ Shapiro},
}

@incollection {itvek,
    AUTHOR = {Berglez, P.},
     TITLE = {On the solutions of the iterated {B}ers-{V}ekua equation},
 BOOKTITLE = {Functional-analytic and complex methods, their interactions,
              and applications to partial differential equations ({G}raz,
              2001)},
     PAGES = {266--275},
 PUBLISHER = {World Sci. Publ., River Edge, NJ},
      YEAR = {2001},
   MRCLASS = {30G20},
  MRNUMBER = {1893260},
MRREVIEWER = {Heinrich Begehr},
}

@incollection {itvekbvp,
    AUTHOR = {Berglez, P. and Luong, T. T.},
     TITLE = {On a boundary value problem for a class of generalized
              analytic functions},
 BOOKTITLE = {Advances in harmonic analysis and operator theory},
    SERIES = {Oper. Theory Adv. Appl.},
    VOLUME = {229},
     PAGES = {91--99},
 PUBLISHER = {Birkh\"{a}user/Springer Basel AG, Basel},
      YEAR = {2013},
   MRCLASS = {30G20 (35F15 35G15)},
  MRNUMBER = {3060409},
MRREVIEWER = {Wolfgang Tutschke},
       DOI = {10.1007/978-3-0348-0516-2\_5},
       URL = {https://doi.org/10.1007/978-3-0348-0516-2_5},
}

@incollection {BerglezBCPseudo,
    AUTHOR = {Berglez, P.},
     TITLE = {On some classes of bicomplex pseudoanalytic functions},
 BOOKTITLE = {Progress in analysis and its applications},
     PAGES = {81--88},
 PUBLISHER = {World Sci. Publ., Hackensack, NJ},
      YEAR = {2010},
      ISBN = {978-981-4313-16-2; 981-4313-16-5},
   MRCLASS = {30G20},
  MRNUMBER = {2757992},
       DOI = {10.1142/9789814313179\_0011},
       URL = {https://doi.org/10.1142/9789814313179_0011},
}

@incollection {BerglezBCPseudo2,
    AUTHOR = {Berglez, P. and Luong, T. T.},
     TITLE = {On a boundary value problem for a class of generalized
              analytic functions},
 BOOKTITLE = {Advances in harmonic analysis and operator theory},
    SERIES = {Oper. Theory Adv. Appl.},
    VOLUME = {229},
     PAGES = {91--99},
 PUBLISHER = {Birkh\"auser/Springer Basel AG, Basel},
      YEAR = {2013},
      ISBN = {978-3-0348-0515-5; 978-3-0348-0516-2},
   MRCLASS = {30G20 (35F15 35G15)},
  MRNUMBER = {3060409},
MRREVIEWER = {Wolfgang\ Tutschke},
       DOI = {10.1007/978-3-0348-0516-2\_5},
       URL = {https://doi.org/10.1007/978-3-0348-0516-2_5},
}

@article {Straube,
    AUTHOR = {Straube, E. J.},
     TITLE = {Harmonic and analytic functions admitting a distribution
              boundary value},
   JOURNAL = {Ann. Scuola Norm. Sup. Pisa Cl. Sci. (4)},
  FJOURNAL = {Annali della Scuola Normale Superiore di Pisa. Classe di
              Scienze. Serie IV},
    VOLUME = {11},
      YEAR = {1984},
    NUMBER = {4},
     PAGES = {559--591},
      ISSN = {0391-173X,2036-2145},
   MRCLASS = {31B25 (32A40 46F20)},
  MRNUMBER = {808424},
MRREVIEWER = {Steven\ R.\ Bell},
       URL = {http://www.numdam.org/item?id=ASNSP_1984_4_11_4_559_0},
}

@article {WB3,
    AUTHOR = {Blair, W. L.},
     TITLE = {Hardy spaces of meta-analytic functions and the Schwarz boundary value problem},
   JOURNAL = {Complex Anal. Synerg.},
  FJOURNAL = {Complex Analysis and its Synergies},
    VOLUME = {10},
      YEAR = {2024},
    NUMBER = {3},
     PAGES = {Paper No. 13, 9},
      ISSN = {2524-7581,2197-120X},
       DOI = {10.1007/s40627-024-00139-9},
       URL = {https://doi.org/10.1007/s40627-024-00139-9}
}

@article {Bers,
    AUTHOR = {Bers, L.},
     TITLE = {An outline of the theory of pseudoanalytic functions},
   JOURNAL = {Bull. Amer. Math. Soc.},
  FJOURNAL = {Bulletin of the American Mathematical Society},
    VOLUME = {62},
      YEAR = {1956},
     PAGES = {291--331},
      ISSN = {0002-9904},
   MRCLASS = {30.0X},
  MRNUMBER = {81936},
MRREVIEWER = {J.\ Ferrand},
       DOI = {10.1090/S0002-9904-1956-10037-2},
       URL = {https://doi.org/10.1090/S0002-9904-1956-10037-2},
}

@article {WBD,
    AUTHOR = {Blair, W. L. and Delgado, B. B.},
     TITLE = {Hardy {S}paces of {S}olutions to {H}igher-{O}rder {V}ekua
              {E}quations and the {S}chwarz {B}oundary {V}alue {P}roblem},
   JOURNAL = {J. Geom. Anal.},
  FJOURNAL = {Journal of Geometric Analysis},
    VOLUME = {35},
      YEAR = {2025},
    NUMBER = {1},
     PAGES = {Paper No. 13},
      ISSN = {1050-6926,1559-002X},
   MRCLASS = {30H10 (30E25 30G20 35G15 46F20)},
  MRNUMBER = {4819984},
       DOI = {10.1007/s12220-024-01825-0},
       URL = {https://doi.org/10.1007/s12220-024-01825-0},
}

@article {WBD2,
    AUTHOR = {Delgado, B. B. and Blair, W. L.},
     TITLE = {Generalizations of poly-{B}ergman spaces and the {B}itsadze
              equation},
   JOURNAL = {Bol. Soc. Mat. Mex. (3)},
  FJOURNAL = {Bolet\'in de la Sociedad Matem\'atica Mexicana. Third Series},
    VOLUME = {31},
      YEAR = {2025},
    NUMBER = {3},
     PAGES = {Paper No. 132},
      ISSN = {1405-213X,2296-4495},
   MRCLASS = {30G20 (30E25 30H20 35J46 46E20)},
  MRNUMBER = {4969552},
       DOI = {10.1007/s40590-025-00812-x},
       URL = {https://doi.org/10.1007/s40590-025-00812-x},
}

@article{BCAtomic,
author = {Blair, W. L.},
title = {An atomic representation for bicomplex Hardy classes},
journal = {Complex Variables and Elliptic Equations},
year = {2025},
doi = {10.1080/17476933.2025.2549391},

}

@article {BCHoiv,
    AUTHOR = {Blair, W. L.},
     TITLE = {Bicomplex {H}ardy classes of solutions to higher-order {V}ekua
              equations},
   JOURNAL = {J. Math. Anal. Appl.},
  FJOURNAL = {Journal of Mathematical Analysis and Applications},
    VOLUME = {558},
      YEAR = {2026},
    NUMBER = {2},
     PAGES = {Paper No. 130397},
      ISSN = {0022-247X,1096-0813},
   MRCLASS = {30G20 (46F20)},
  MRNUMBER = {5016713},
       DOI = {10.1016/j.jmaa.2026.130397},
       URL = {https://doi.org/10.1016/j.jmaa.2026.130397},
}

@unpublished{BCHarmVek,
AUTHOR = {Blair, W. L.},
TITLE = {An atomic representation for Hardy spaces of generalized analytic functions},
NOTE = {Submitted}
}

@unpublished{BCSchwarz,
AUTHOR = {Blair, W. L.},
TITLE = {Bicomplex Schwarz and Dirichlet Boundary Value Problems},
NOTE = {Submitted. URL: https://arxiv.org/abs/2507.15653}
}

@unpublished{BCBeltrami,
AUTHOR = {Blair, W. L.},
TITLE = {Bicomplex Hardy Classes of Solutions to Beltrami Equations and the Schwarz Boundary Value Problem},
NOTE = {Submitted. URL: https://arxiv.org/abs/2507.15657}
}

@article {Segre,
    AUTHOR = {Segre, C.},
     TITLE = {Le rappresentazioni reali delle forme complesse e gli enti
              iperalgebrici},
   JOURNAL = {Math. Ann.},
  FJOURNAL = {Mathematische Annalen},
    VOLUME = {40},
      YEAR = {1892},
    NUMBER = {3},
     PAGES = {413--467},
      ISSN = {0025-5831,1432-1807},
   MRCLASS = {99-04},
  MRNUMBER = {1510728},
       DOI = {10.1007/BF01443559},
       URL = {https://doi.org/10.1007/BF01443559},
}

@article {CastaKrav,
    AUTHOR = {Casta\~neda, A. and Kravchenko, V. V.},
     TITLE = {New applications of pseudoanalytic function theory to the
              {D}irac equation},
   JOURNAL = {J. Phys. A},
  FJOURNAL = {Journal of Physics. A. Mathematical and General},
    VOLUME = {38},
      YEAR = {2005},
    NUMBER = {42},
     PAGES = {9207--9219},
      ISSN = {0305-4470,1751-8121},
   MRCLASS = {35Q40 (30G20 35C05)},
  MRNUMBER = {2186602},
       DOI = {10.1088/0305-4470/38/42/003},
       URL = {https://doi.org/10.1088/0305-4470/38/42/003},
}

@book {PriceMulti,
    AUTHOR = {Price, G. B.},
     TITLE = {An introduction to multicomplex spaces and functions},
    SERIES = {Monographs and Textbooks in Pure and Applied Mathematics},
    VOLUME = {140},
      NOTE = {With a foreword by Olga Taussky Todd},
 PUBLISHER = {Marcel Dekker, Inc., New York},
      YEAR = {1991},
     PAGES = {xiv+402},
      ISBN = {0-8247-8345-X},
   MRCLASS = {30G35 (32A30)},
  MRNUMBER = {1094818},
MRREVIEWER = {John\ Ryan},
}

@article {BCTransmutation,
    AUTHOR = {Vicente-Ben\'itez, V. A.},
     TITLE = {Transmutation operators and complete systems of solutions for
              the radial bicomplex {V}ekua equation},
   JOURNAL = {J. Math. Anal. Appl.},
  FJOURNAL = {Journal of Mathematical Analysis and Applications},
    VOLUME = {536},
      YEAR = {2024},
    NUMBER = {2},
     PAGES = {Paper No. 128224, 28},
      ISSN = {0022-247X,1096-0813},
   MRCLASS = {30G20 (35J46 46E20)},
  MRNUMBER = {4708695},
       DOI = {10.1016/j.jmaa.2024.128224},
       URL = {https://doi.org/10.1016/j.jmaa.2024.128224},
}

@article {BCBergman,
    AUTHOR = {Vicente-Ben\'itez, V. A.},
     TITLE = {Bergman spaces for the bicomplex {V}ekua equation with bounded
              coefficients},
   JOURNAL = {J. Math. Anal. Appl.},
  FJOURNAL = {Journal of Mathematical Analysis and Applications},
    VOLUME = {543},
      YEAR = {2025},
    NUMBER = {2},
     PAGES = {Paper No. 129025},
      ISSN = {0022-247X,1096-0813},
   MRCLASS = {30 (32 46E22)},
  MRNUMBER = {4822609},
       DOI = {10.1016/j.jmaa.2024.129025},
       URL = {https://doi.org/10.1016/j.jmaa.2024.129025},
}

@article {FundBicomplex,
    AUTHOR = {Campos, H. M. and Kravchenko, V. V.},
     TITLE = {Fundamentals of bicomplex pseudoanalytic function theory:
              {C}auchy integral formulas, negative formal powers and
              {S}chr\"odinger equations with complex coefficients},
   JOURNAL = {Complex Anal. Oper. Theory},
  FJOURNAL = {Complex Analysis and Operator Theory},
    VOLUME = {7},
      YEAR = {2013},
    NUMBER = {2},
     PAGES = {485--518},
      ISSN = {1661-8254,1661-8262},
   MRCLASS = {30G35 (30D60 30E20)},
  MRNUMBER = {3037061},
MRREVIEWER = {John\ Ryan},
       DOI = {10.1007/s11785-012-0256-4},
       URL = {https://doi.org/10.1007/s11785-012-0256-4},
}

@book {BCHolo,
    AUTHOR = {Luna-Elizarrar\'as, M. E. and Shapiro, M. and Struppa,
              D. C. and Vajiac, A.},
     TITLE = {Bicomplex holomorphic functions},
    SERIES = {Frontiers in Mathematics},
      NOTE = {The algebra, geometry and analysis of bicomplex numbers},
 PUBLISHER = {Birkh\"auser/Springer, Cham},
      YEAR = {2015},
     PAGES = {viii+231},
      ISBN = {978-3-319-24866-0; 978-3-319-24868-4},
   MRCLASS = {30-02 (30G35 32A30)},
  MRNUMBER = {3410909},
MRREVIEWER = {Alessandro\ Perotti},
       DOI = {10.1007/978-3-319-24868-4},
       URL = {https://doi.org/10.1007/978-3-319-24868-4},
}

@article {ComplexSchr,
    AUTHOR = {Rochon, D.},
     TITLE = {On a relation of bicomplex pseudoanalytic function theory to
              the complexified stationary {S}chr\"odinger equation},
   JOURNAL = {Complex Var. Elliptic Equ.},
  FJOURNAL = {Complex Variables and Elliptic Equations. An International
              Journal},
    VOLUME = {53},
      YEAR = {2008},
    NUMBER = {6},
     PAGES = {501--521},
      ISSN = {1747-6933,1747-6941},
   MRCLASS = {30G35 (30G20)},
  MRNUMBER = {2421815},
MRREVIEWER = {Baruch\ A.\ Schneider},
       DOI = {10.1080/17476930701769058},
       URL = {https://doi.org/10.1080/17476930701769058},
}

@article{Dirichletbianalytic,
author = {Moreno García, A.   and  Villaseñor, J. R. and  Abreu-Blaya, R. and  Sánchez Santiesteban, J. L. },
title = {On the Dirichlet problem for bianalytic functions},
journal = {Complex Variables and Elliptic Equations},
volume = {0},
number = {0},
pages = {1--12},
year = {2025},
publisher = {Taylor \& Francis},
doi = {10.1080/17476933.2025.2502958},
URL = {    
        https://doi.org/10.1080/17476933.2025.2502958
},
eprint = { 
    
        https://doi.org/10.1080/17476933.2025.2502958}}

@book {GK,
    AUTHOR = {Greene, R. E. and Krantz, S. G.},
     TITLE = {Function theory of one complex variable},
    SERIES = {Graduate Studies in Mathematics},
    VOLUME = {40},
   EDITION = {Third},
 PUBLISHER = {American Mathematical Society, Providence, RI},
      YEAR = {2006},
     PAGES = {x+504},
      ISBN = {0-8218-3962-4},
   MRCLASS = {30-01},
  MRNUMBER = {2215872},
       DOI = {10.1090/gsm/040},
       URL = {https://doi.org/10.1090/gsm/040},
}

@article{Bitsadze1948,
  title={About the uniqueness of the Dirichlet problem for elliptic
partial differential equations (in Russian)},
  author={Bitsadze, A.~V.},
  journal={Uspekhi Mat. Nauk 3},
  volume={6},
  number={28},
  year={1948}
}

@article {Beg04,
    AUTHOR = {Begehr, H.},
     TITLE = {Boundary value problems for the {B}itsadze equation},
   JOURNAL = {Mem. Differential Equations Math. Phys.},
  FJOURNAL = {Georgian Academy of Sciences. A. Razmadze Mathematical
              Institute. Memoirs on Differential Equations and Mathematical
              Physics},
    VOLUME = {33},
      YEAR = {2004},
     PAGES = {5--23},
      ISSN = {1512-0015},
   MRCLASS = {30G20 (30E20 30E25 35C15 35J25)},
  MRNUMBER = {2119896},
MRREVIEWER = {Vladimir\ Mityushev},
}

@article {Karaca20,
    AUTHOR = {Karaca, B.},
     TITLE = {Dirichlet problem for complex model partial differential
              equations},
   JOURNAL = {Complex Var. Elliptic Equ.},
  FJOURNAL = {Complex Variables and Elliptic Equations. An International
              Journal},
    VOLUME = {65},
      YEAR = {2020},
    NUMBER = {10},
     PAGES = {1748--1762},
      ISSN = {1747-6933,1747-6941},
   MRCLASS = {30G20 (35J56 35J57 35J58)},
  MRNUMBER = {4142781},
MRREVIEWER = {Ahmet\ Okay\ \c Celebi},
       DOI = {10.1080/17476933.2019.1684478},
       URL = {https://doi.org/10.1080/17476933.2019.1684478},
}

\end{document}